\documentclass[10pt]{article}
\expandafter\let\csname equation*\endcsname\relax
\expandafter\let\csname endequation*\endcsname\relax
\usepackage{bm}
\usepackage{enumerate}
\usepackage[top=1in, bottom=1in, left=1in, right=1in]{geometry}
\usepackage[dvipsnames]{xcolor}
\usepackage{mathrsfs}
\usepackage{amsfonts}
\usepackage{multirow}
\usepackage{amssymb,dsfont}
\usepackage{caption,comment}
\usepackage{subcaption}
\usepackage{blkarray}
\usepackage{amsmath}
\usepackage{float}
\usepackage{amssymb,amsthm}
\usepackage[toc,page]{appendix}
\usepackage{graphicx,afterpage}
\usepackage{epstopdf}
\usepackage[normalem]{ulem}
\usepackage{authblk}
\usepackage[pdftex,colorlinks=true]{hyperref}
\usepackage[most]{tcolorbox}
\hypersetup{CJKbookmarks,%
	bookmarksnumbered,%
	colorlinks,%
	linkcolor=black,%
	citecolor=black,%
	plainpages=false,%
	pdfstartview=FitH}
\numberwithin{equation}{section}
\numberwithin{figure}{section}

\newcommand\tabcaption{\def\@captype{table}\caption}

\newtheorem{thm}{Theorem}[section]

\newtheorem{lem}[thm]{Lemma}

\newtheorem{defn}[thm]{Definition}

\newtheorem{rem}[thm]{Remark}
\theoremstyle{remark}

\newcommand{\tr}{\mathrm{trace}}

\makeatletter
\DeclareFontFamily{OMX}{MnSymbolE}{}
\DeclareSymbolFont{MnLargeSymbols}{OMX}{MnSymbolE}{m}{n}
\SetSymbolFont{MnLargeSymbols}{bold}{OMX}{MnSymbolE}{b}{n}
\DeclareFontShape{OMX}{MnSymbolE}{m}{n}{
    <-6>  MnSymbolE5
   <6-7>  MnSymbolE6
   <7-8>  MnSymbolE7
   <8-9>  MnSymbolE8
   <9-10> MnSymbolE9
  <10-12> MnSymbolE10
  <12->   MnSymbolE12
}{}
\DeclareFontShape{OMX}{MnSymbolE}{b}{n}{
    <-6>  MnSymbolE-Bold5
   <6-7>  MnSymbolE-Bold6
   <7-8>  MnSymbolE-Bold7
   <8-9>  MnSymbolE-Bold8
   <9-10> MnSymbolE-Bold9
  <10-12> MnSymbolE-Bold10
  <12->   MnSymbolE-Bold12
}{}

\let\llangle\@undefined
\let\rrangle\@undefined
\DeclareMathDelimiter{\llangle}{\mathopen}%
                     {MnLargeSymbols}{'164}{MnLargeSymbols}{'164}
\DeclareMathDelimiter{\rrangle}{\mathclose}%
                     {MnLargeSymbols}{'171}{MnLargeSymbols}{'171}
\makeatother

\newcommand{\cP}{\mathcal P}

\newcommand{\cC}{\mathcal C}
\newcommand{\cM}{\mathcal M}
\newcommand{\cN}{\mathcal N}
\newcommand{\cL}{\mathcal L}

\newcommand{\dd }{\mathop{}\!\mathrm{d}}
\newcommand{\DD }{\mathop{}\!\mathrm{D}}
\newcommand{\ddt }{\dd t}
\newcommand{\R}{\mathbb R}
\renewcommand{\P}{\mathbb P}

\newcommand{\eps}{\varepsilon}
\newcommand{\E}{\mathbb E}

\newcommand{\greybox}[1]{
\begin{tcolorbox}[
  colback=gray!10,
  colframe=gray!50,
  arc=4mm,
  boxrule=0.5pt,
  left=2mm, right=2mm, top=1mm, bottom=1mm,
  enhanced,
  breakable,     
  title=\textbf{Guiding example}, 
  fonttitle=\normalsize\bfseries, 
  coltitle=black           
]
#1
\end{tcolorbox}
}

\title{An optimal experimental design approach to sensor placement in continuous stochastic filtering}
\author{Sahani Pathiraja, Claudia Schillings, Philipp Wacker}
\date{January 2025}

\allowdisplaybreaks

\begin{document}

\maketitle

\begin{abstract}
    Sequential filtering and spatial inverse problems assimilate data points distributed either temporally (in the case of filtering) or spatially (in the case of spatial inverse problems). Sometimes it is possible to choose the position of these data points (which we call sensors here) in advance, with the goal of maximising the expected information gain (or a different metric of performance) from future data, and this leads to an Optimal Experimental Design (OED) problem. Here we revisit an interpretation of optimising sensor placement as an integration with respect to a general probability measure $\xi$. This generalises the problem of discrete-time sensor placement (which corresponds to the special case where the probability measure is a mixture of Diracs)  to an infinite-dimensional, but mathematically more well-behaved setting. We focus on the continuous-time stochastic filtering setting, whose solution is governed by the Zakai equation.  We derive an expression for the Fr\'echet derivative of a general OED utility functional, the key to which is an adjoint (backwards in time) differential equation. This paves the way for utilising new gradient-based methods for solving the corresponding optimisation problem, as a potentially more efficient alternative to (semi-)discrete optimisation methods, e.g. based on greedy insertion and deletion of sensor placements.
\end{abstract}

\tableofcontents

\section{Main idea, prior work, and contributions}
This work focuses on parameter estimation problems with either a temporal dependency (a filtering or data assimilation problem) or a spatial varying parameter (more commonly called an inverse problem). Often, sensors are associated to positions in the temporal or spatial domain, and the problem of optimal sensor placement encompasses finding the best set of locations at which to take measurements. The most common approach is to directly work with the sensor locations $t_i$, 
either by solving a high-dimensional optimisation problem with difficult to use symmetry (under swapping the position of any two sensors), often using sparsity-inducing regularisation (such as $\ell^0$ or $\ell^1$) \cite{alexanderian2014optimal,haber2008numerical,haber2012numerical,ucinski2020d,das2020optimal}, by relaxing to a convex optimisation problem as in \cite{maity2022sensor}, by choosing a best subset (which is an NP-hard combinatoric problem \cite{ye2018complexity}), or greedy algorithms based on insertion and deletion \cite{aarset2024global,de2024ensemble,mula2025dynamical}. Alternative approaches include dynamical approximation \cite{mula2025dynamical} or particle-based methods for adaptive experimental design in \cite{iqbal2024nesting}, following ideas proposed more generally for experimental design in general in \cite{amzal2006bayesian}. 
Related work on sensor design includes \cite{zinage2022optimal}, where the authors directly optimise the observation operator / sensor gain in order to solve the Optimal experimental design (OED) problem for a linear-Gaussian Kalman-Bucy filter system. An alternative approach to the route taken here is to think of experimental design as a random variable (such as a Poisson process) from which measurement instances are sampled, as in \cite{xiong2023bias}. In situations where the signal dynamics is subject to active control, experimental design can be extended to the task of finding an optimal control policy, as in \cite{hooker2015control}. This provides a connection to optimal control, albeit in a way very different from our proposed approach.
 
Sensor placement is a specific case of an optimal experimental design problem. This is a fairly established field and an overview of historical research can be found in several textbooks and (review) articles, such as \cite{ucinski2004optimal,muller2004optimal,muller2005simulation,pukelsheim2006optimal,atkinson2007optimum,huan2014gradient,alexanderian2021optimal,huan2024optimal}. 

In this manuscript we draw attention to an idea originally developed in \cite{kiefer1952stochastic}, which is to relax the setting of evaluation in discrete sensor locations to integration with respect to a probability measure\footnote{with the former setting being a special case by choosing an empirical probability measure supported on finitely many locations, $\xi = \tfrac1N\sum_{i=1}^N \delta_{t_i}$} $\xi$. See figure \ref{fig:main_idea} for an illustration of this idea. This perspective has been revisited a few times, such as in the review paper \cite{huan2024optimal}, and recently for linear (frequentist) statistical inverse problems in \cite{jin2024optimal}. We note that ``sensor scheduling'' and related terms are not well-defined and different communities mean very different things. For example, we do \textit{not} consider the setting where there is a given set of sensors, of which only one can be active at a time (and optimal sensor scheduling means finding the optimal switching procedure), such as in \cite{maity2022sensor}.
\textbf{Our contributions are: }
\begin{itemize}
    \item We introduce measure-based sensor design in the context of stochastic filtering. 
    \item We develop a novel formulation of log-Zakai-type stochastic filtering based on this relaxation, where the design variable $\xi$ enters both drift and diffusion terms in a natural way.
    \item We derive an algorithmic expression for the Fr\'echet derivative of general utility functionals used for optimal sensor placement, based on the adjoint approach to gradients, which leads to a (backwards in time) adjoint differential equation with a structure similar to (and computational effort comparable to) the Zakai equation.
    \item As a special case, we derive (in the linear-Gaussian case), an equivalent expression for the Kalman-Bucy filtering equations.
    \item With this, we open the door to efficient gradient based optimisation algorithms, and we point out several ways in which the present work can (and should) be extended. 
\end{itemize}
The calculations in this manuscript are mostly formal, with rigorous justification of e.g. the well-posedness of the adjoint equation omitted, as the main purpose of this work is to present a new approach to OED in the continuous time filtering setting.
\begin{figure}
    \centering
    \includegraphics[width=\linewidth]{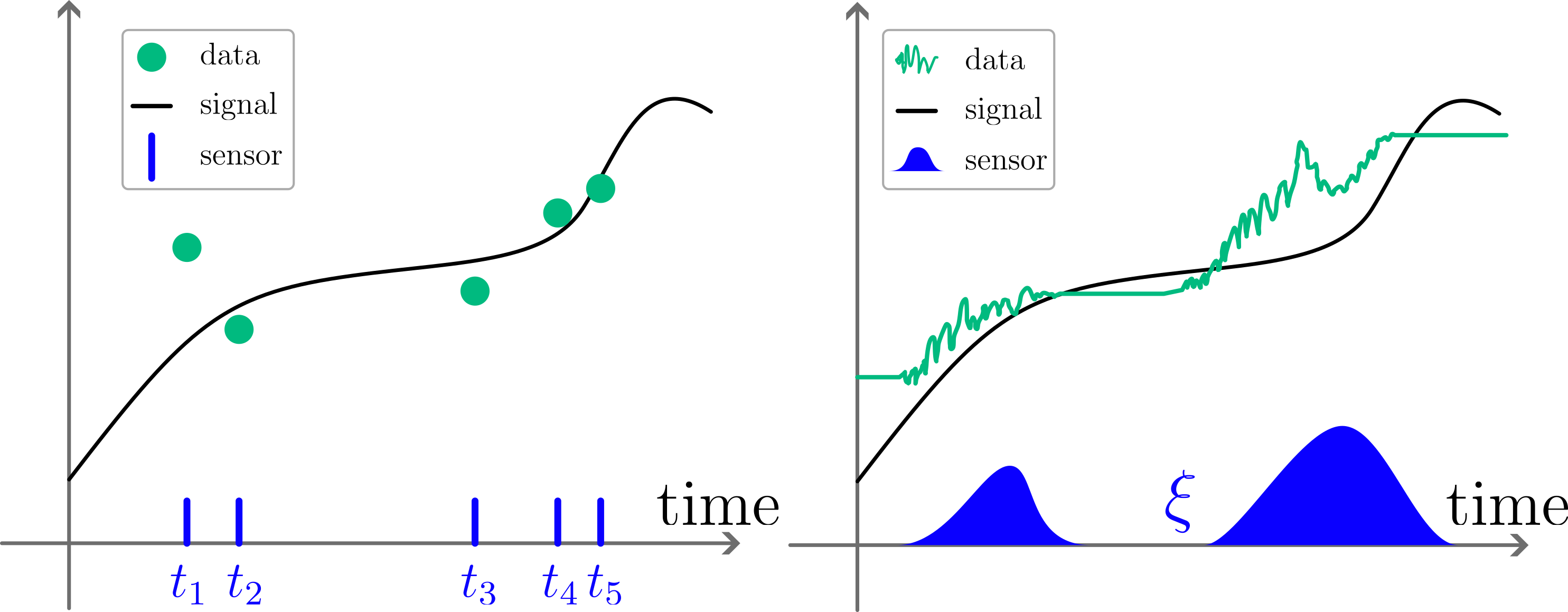}
    \caption{Left: Sensor placement at discrete time points $\{t_i\}$. Right: Sensor setup corresponds to temporal integration with respect to a probability measure $\xi$.}
    \label{fig:main_idea}
\end{figure}

We begin with a simple motivational example before discussing the necessary mathematical prerequisites.

\subsection{Setup of guiding example: Monitoring logistic growth}
We will use the following simple one-dimensional parameter estimaton problem as a guiding example throughout the manuscript, as a way of motivating ideas and notation.

\greybox{
We consider an unknown parameter which we model as a time-independent state $X(t) = x_0$, which we interpret as the rate of an logistic growth process of a concentration $Z(t) = g(X(t), t) = \frac{\exp(x_0 t)}{z_0^{-1} - 1 + \exp(x_0 t)}$, where the observation map is $g(x, t) = \frac{\exp(x t)}{z_0^{-1} - 1 + \exp(x t)}$, with a fixed and known $z_0 \in (0,1)$. 

Very generally, our goal will be to find the best experimental setup such that we can learn as much as possible about the value of the unknown growth parameter $x_0$. As a starting point (although we will generalise this considerably), we assume that we collect exactly one observation of form 
\begin{equation}\label{eq:log_dirac}
    Y = g(X_\tau, \tau) + \eps,
\end{equation}
where $\tau$ is the time at which we measure the concentration $z$. This problem has a closed form solution: If we use D-optimality (see, e.g., \cite{pukelsheim2006optimal}) as a criterion for optimal experimental design, we should be choosing $\tau = \frac{\log (z_0^{-1}-1)}{x_0}$ (note that this ignores prior information on $x_0$). Intuitively, this is the time at which the ``fan of trajectories'' spreads out the most, so pointwise observation at this time is bound to give a high amount of information about the underlying parameter.

    \includegraphics[width=0.75\linewidth]{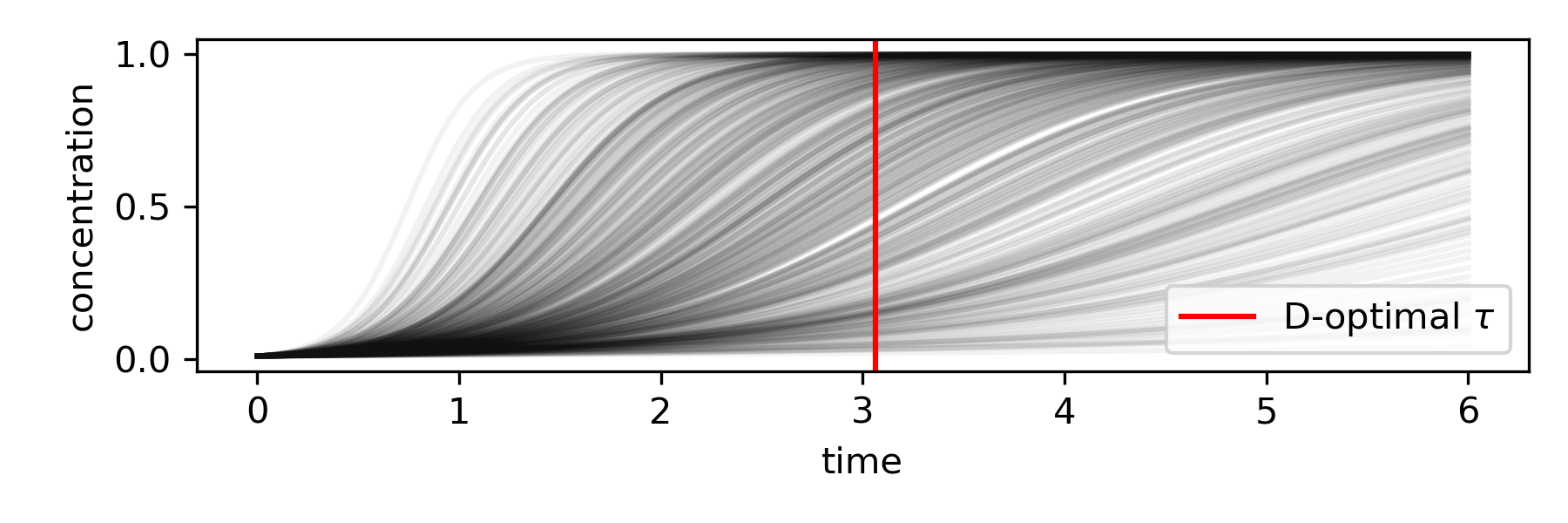}

This approach does not generalise well to optimising several evaluation times $\tau_i$ where we collect data $Y_i = g(X_{\tau_i},\tau_i) + \eps_i$ simultaneously, and becomes either a high-dimensional problem with a lot of unusable symmetry (since permutation of the $\tau_i$ leads to the same outcome), or to a tricky discrete optimisation problem (if we define a finite set of viable time points $T$ in advance, and we are to pick an optimal subset). 

Our goal in this manuscript is to relax the data acquisition process from the discrete case of individual observations to continuous-time monitoring, but allowing for a time-inhomogeneous ``sensor density''. Put succinctly, we start with \eqref{eq:log_dirac}, which can be thought of as an observation model with sensor schedule $\xi = \delta_\tau$, a Dirac measure at position $t=\tau$. Then we are going to generalise this to the case of several pointwise observations $Y_i = g(X_{t_i}, t_i) + \eps_i$ (with possibly variable observational noise strength depending on a parameter $\alpha_i$), which corresponds to a sensor distribution of $\xi = \sum_i \alpha_i \delta_{t_i}$, an empirical probability measure at times $t_i$ such that its total mass $\sum_i \alpha_i = 1$ sums to one. Finally, we will replace $\xi$ by a general probability distribution, usually assuming that $\xi$ has a Lebesgue density. We think of $\xi$ as the experimental design variable and will call it the \textbf{sensor schedule}. The fact that $\xi$ is a probability measure represents bounded resources. After setting up an optimality criterion (measuring how well the experiment allows us to learn about underlying parameters or states, depending on the context), this will allow us to formulate and solve the associated optimalisation problem of finding the optimal sensor schedule $\xi$.
}

Before discussing the continuous time case, we remind ourselves of a few basic concepts from stochastic filtering, see for example \cite{bain2009fundamentals} for a comprehensive overview.

\subsection{Background on Stochastic filtering}
We denote all variables in this section with a tilde, such as $\tilde X(t)$, since we will derive modified versions of these formulae later on.
\paragraph{Nonlinear continuous-time filtering problem}
We consider a latent stochastic process $\tilde X(t)$ (the signal) with infinitesimal generator $\cL$ and data (the observation) given by 
\begin{equation}\label{eq:filtering_cts_guiding}
    \dd \tilde Z(t) = g(\tilde X(t))\dd t + \,\Gamma^{1/2}(t)\dd W_t.
\end{equation}
The conditional distribution of $\tilde X(t)$ given the path $\{\tilde Z(s)\}_{s\leq t}$ is given by $\tilde q_t$ and evolves according to the Kushner-Stratonovich equation
\begin{align*}
    \dd \tilde q_t(x) &= (\cL^\star \tilde q_t)(x)\dd t + \tilde q_t(x)\left(g(x) - \bar g_t \right)^\top \Gamma^{-1}(t)\left(\dd \tilde Z(t) -\bar g_t\dd t\right),
\end{align*}
where $\bar g_t = \int g(x)q_t(x)\dd x$. It is often easier (computationally less intensive and numerically more stable) to work with an unnormalised density $\tilde{p}_t$ instead, which is governed by the Zakai equation
\begin{equation}\label{eq:zakai}
    \dd \tilde p_t(x) = (\cL^\star \tilde p_t)(x)\dd t + \tilde p_t(x) g(x)^\top \Gamma^{-1}(t)\dd \tilde Z(t),
\end{equation}
where we can always compute the normalised density via $\tilde q_t(x) = \tilde p_t(x)/\int \tilde p_t(y)\dd y$ from $\tilde p_t$. We find that numerical stability can be improved further by tracking the unnormalised log-density: This makes sure that (taking the exponential), $\tilde p_t(x)$ is always non-negative, which is not the case for naive implementations of \eqref{eq:zakai}. Using a simple It\=o-type computation, we have $
    \dd \log \tilde p_t(x) = {p_t(x)}^{-1}{\dd\tilde p_t(x)} -\tfrac{1}{2}{p_t(x)}^{-2}{(\dd p_t(x))^2}$, so
\begin{equation}\label{eq:logzakai}
    \dd \log \tilde p_t(x) = \frac{\cL^\star  \tilde p_t(x)}{\tilde p_t(x)}\dd t - \frac{1}{2}\left\|g(x)\right\|_{\Gamma(t)}^2\dd t +g(x)^\top \Gamma^{-1}(t)\dd \tilde Z(t). 
\end{equation}
It is also possible to derive a similar form of the logarithmic Kushner-Stratonovich equation, which is omitted here.

\paragraph{Linear continuous-time filtering problem.}
If the signal and observation dynamics have linear and Gaussian structure,
\begin{equation}\label{eq:linear_Gauss_unweighted}
\begin{split}
    \dd \tilde X(t) &= L \tilde X(t)\dd t + \Sigma^{1/2}(t)\dd B_t\\
    \dd \tilde Z(t) &= H \tilde X(t)\dd t + \Gamma^{1/2}(t)\dd W_t,
\end{split}
\end{equation}
then the conditional distribution of $\tilde X(t)$ given the path $\{\tilde Z(s)\}_{s\leq t}$ is given by $\tilde p_t = \cN(\tilde m(t), \tilde C(t))$, where
\begin{equation}\label{eq:Kalman_Bucy}
\begin{split}
    \dd \tilde m(t) &= L \tilde m(t) + \tilde C(t) H^\top\Gamma(t)^{-1}\left(\dd \tilde Z(t)-H\tilde m(t)\dd t\right)\\
    \frac{\dd \tilde C(t)}{\dd t} &=L \tilde C(t) + \tilde C(t) L^\top + \Sigma(t) - \tilde C(t) H^\top \Gamma^{-1}(t) H \tilde C(t)
\end{split}
\end{equation}
These are the celebrated Kalman-Bucy filtering equations. 

\paragraph{Nonlinear and linear discrete-time filtering problem.}
We can also accommodate the setup where observations are not collected continuously, but at discrete time steps: In this case we will use $y_i$ for the discrete measurements obtained via
\begin{equation}\label{eq:filtering_disc}
    Y_{t_i} = g(X_{t_i}) +  \Gamma^{1/2}(t_i)\eps_{t_i}.
\end{equation}
Filtering is then done by evolving the filtering distribution according to the Fokker-Planck equation, and then assimilating data momentarily (and discontinuously) as they arrive, via Bayes' law. If the signal and observation dynamics are linear, then we recover the usual Kalman (discrete-time) filter.

\subsection{Sensor placement as an optimal experimental design task}
All of these filtering settings mentioned take for granted a given way of information acquisition. Either, continuously (in the case of continuous-time filtering in the form $\dd \tilde Z(t) = g(\tilde X(t))\dd t + \Gamma^{1/2}(t)\dd W_t$), or at given discrete timepoints (in the case of discrete-time filtering in the form $Y_{i} = g(\tilde X_{t_i}) + \Gamma^{1/2}(t_i)\eps_i$).

We can imagine, though, that conscious design of the observational setup (or schedule) can provide more informative data, e.g. by contracting the filtering distribution more strongly.

If observations arrive at discrete time steps, it is conceivable that there are time steps which are more suitable than others. An obvious example for this is when the observational noise covariance $\Gamma(t)$ is time-dependent: If we pick observational time points $t_i$ where $\Gamma(t_i)$ is small, we will likely get more informative data. Or, in our guiding example of logistic growth, observation times $\tau\approx 0$ or $\tau \gg 10$ are likely to not provide much information, either, since $g(X_\tau,\tau)$ is essentially a deterministic value ($0$, or $1$, respectively) regardless of the value of $x_0$.

If observation is performed continuously, then this is less obvious: How can we observe ``better than continuously?'' We need to keep in mind, though, that even continuous observation of the form \eqref{eq:filtering_cts} does not translate into perfect information: Data is collected continuously, but it is simultaneously perturbed by measurement noise. We can now imagine concentrating our ``observational attention'' towards times where, for example, $\Gamma(t)$ is small in comparison (or, as in our guiding example, towards times where we expect to learn more about the signal process). This leads to data collection \textbf{more strongly concentrated} on some intervals than others. We can imagine changing from ``uniform observation'' (which is the standard set-up for stochastic filtering) to ``observation with variable sensor strength'' (mentally replacing a uniform distribution with, e.g., a Gaussian distribution). We model variability in sensor strength with a probability measure $\xi$, which we call the \textbf{sensor schedule}. Intuitively, $\int_s^t \dd\xi$ quantifies the fraction of ``sensor budget'' allocated to the time interval $[s,t]$. 

\subsection{Structure of the Manuscript}
In section \ref{sec:derivation_weighted} we give more details on how this allows us to derive \textbf{weighted} filtering equations, arising from observational models where we have an experimental design choice of ``when to observe, and how densely'' (the sensor schedule $\xi$). We will start with the discrete-time setup since this is more instructive, and then informally derive its continuous-time version. We call these \textbf{scheduled filtering equations}, which will yield our main filtering setup subject to which we will be performing experimental design. In short, our plan is as follows: Starting from discrete observations (corresponding to a ``sensor schedule'' $\xi$ being a weighted sum of Dirac measures), we generalise to arbitrary probability measures $\xi$. This allows us to pose a corresponding stochastic filtering problem depending on the sensor schedule $\xi$. In section \ref{sec:setupOED} we frame the search for an optimal sensor schedule $\xi$ as an optimal experimental design problem based on a utility functional $U$ dependent on $\xi$. In section \ref{sec:solvingOED}, we show how the technique of adjoints allows us to derive an expression for the derivative of $U$ with respect to $\xi$, which we can translate into a gradient using a suitable geometry, and perform an optimisation procedure towards finding the optimal sensor schedule $\xi$. Throughout the manuscript, we will keep revisiting our simple guiding example to illustrate these ideas in a friendly example, while keeping the theory fairly general. We close with some numerical experiments in section \ref{sec:numerics}, demonstrating that everything works as expected, and conclude with remarks on insights we got, as well as comment on opportunity for future work in \ref{sec:conclusion}. We begin now with a short overview of the state of research in Optimal Experimental Design, with a focus on the problems we are interested in this manuscript, and fix notation.

\subsection{Notation}
\begin{itemize}
    \item $(\cM,\llangle \cdot,\cdot\rrangle)$ is the space of square matrices $\R^{n\times n}$ (with the dimensionality usually clear from context), together with the Frobenius inner product: For $M,N\in \cM$, we set $\llangle M,N\rrangle= \sum_{ij}M_{ij}N_{ij}  = \tr(M^\top N)$. We denote the subset of symmetric and positive definite matrices (covariance matrices) as $\cM^+$.
    \item Let $U$ be a vector space. Then  $U^\star$ is the vector space of bounded linear maps in $\cL(U,\R)$. We will sometimes write dual action as an inner product, i.e. for $\phi\in U^\star$, we may write $\phi(u) = \langle \phi,u\rangle_U$.
    \item If $A : U\to V$ is a linear operator, then $A^\star: V^\star \to U^\star$ is the adjoint operator defined via $(A^\star \psi)(u) = \psi(Au)$ for $u\in U$, $\psi\in 
    V^\star$.
    \item For a map of the form $\cL(\dot C, C, \xi, t)$, we denote by $\partial_{\dot C}\cL(\dot C, C, \xi, t)$ the partial derivative of $\cL$ in $(\dot C, C, \xi, t)$ with respect to the first component $\dot C$. Analogously, we define $\partial_C \cL, \partial_\xi \cL$, and $\partial_t \cL$.
    \item If $f:S\to V$, $f: x\mapsto f(x)$ with $S\subset U$ is a map, then $\DD_x f(x)[h]$ is the Fr\'echet derivative of $f$ in $x$ in direction $h$. 
\end{itemize}
\section{Optimal sensor design for stochastic filtering}

\subsection{Preamble: Allocating a sensor budget}
To get ourselves familiar with the general idea of sensor placement, let's consider the following straightforward setup where we are trying to optimise spending a ``sensor budget'' between two fixed sources of data. Assume we want to infer the value of an unknown parameter $X\in \R$ and we have two ways of getting information about $X$,
\begin{align*}
    W_1 &= X + \frac{\gamma_1}{\sqrt{\alpha_1}}\eps_1\\
    W_2 &= X + \frac{\gamma_2}{\sqrt{\alpha_2}}\eps_2.
\end{align*}
This means that we assume that $\eps_i\sim\cN(0,1)$ is standard normal observation error, and the data $W_i$ arises from perturbing the unknown $X$ with additive Gaussian noise scaled by a standard deviation $\gamma_i$ depending on $i$, and additionally on a term $ \alpha_i$ that models a budgetary constraint, where we say that 
\begin{equation}\alpha_i\geq 0\quad \sum_i \alpha_i = 1. \label{eq:budget}    
\end{equation}
Let's for concreteness assume that $\gamma_1 < \gamma_2$, i.e., the first datapoint is inherently more reliable than the second datapoint.
Obviously, letting both $\alpha_i\to+\infty$ would be optimal from a statistical perspective (since $W_i\to X$ is going to converge to a perfect observation of $X$), but the budget constraints forbid this.  One way to interpret this budgetary constraint is (noting that $\alpha_i^{-1}$ is a variance-type quantity, so $\alpha_i$ is a precision-type quantity) as an upper bound on \textit{total precision}.

The question now is, \textit{how should we set $\alpha_i$} (in alignment with the budgetary constraints \eqref{eq:budget}) such that our information about $X$ is maximised. One way to think about the quality of data is the signal-to-noise ratio (SNR), which in this case (up to a constant) is $\gamma_i/\sqrt{\alpha_i}$ for the $i$-th datapoint.

The most straightforward way of solving this question is to make sure that the likelihood $\P(W_1,W_2|X)$ is as sharply peaked as possible. A quick calculation using quadratic expansion shows that
\begin{align*}
    \P(W_1,W_2|X) &= \P(W_1|X)\P(W_2|X) = \exp\left(-\frac{1}{2}\left[\frac{\alpha_1}{\gamma_1^2} + \frac{\alpha_2}{\gamma_2^2}\right]\,\left(X - \frac{\alpha_1/\gamma_1^2 W_1 + \alpha_2/\gamma_2^2W_2}{\frac{\alpha_1}{\gamma_1^2} + \frac{\alpha_2}{\gamma_2^2}}\right)^2  \right).
\end{align*}
Making the likelihood as sharply peaked as possible amounts to the maximisation problem
\begin{align*}
    \max \frac{\alpha_1}{\gamma_1^2} + \frac{\alpha_2}{\gamma_2^2},\quad \alpha_i\geq 0,\, \alpha_1+\alpha_2 = 1,
\end{align*}
which can be seen to be solved by $\alpha_1 = 1$ and $\alpha_2 = 0$ (because of our assumption that $\gamma_1 < \gamma_2$).

The way $W_i$ is defined is slightly inconvenient: For $\alpha_2\to 0$ (which is the optimal choice), $W_2 \to \pm\infty$ (with unclear sign). By multiplying both sides by a factor of $\alpha_i$, we get the following, completely equivalent, model:
\begin{equation}\label{eq:rescaled_model}
    Y_i := \alpha_i W_i = \alpha_i X+ \gamma_i \sqrt{\alpha_i}\eps_i,
\end{equation}
which doesn't show this behaviour. Since $\alpha_i$ is, at time of observation, a fixed and deterministic number, it does not matter whether we work with $W_i$ or $Y_i$,\footnote{and the Bayesian posterior would be identical, since $\P(X|W_1,W_2) = \P(X|Y_1,Y_2)$} but $Y_i$ has bounded second moments for all possible budgetary allocations $\alpha_i$. Also, the SNR remains $\gamma_i/\sqrt{\alpha_i}$ under this constant rescaling.

We will now generalise these ideas to the setting where we observe a time-dependent quantity $X(t)$ (rather than a constant value $X$), and then take a continuous-time limit, starting our consideration with the form \eqref{eq:rescaled_model} right away.

\subsection{Derivation of filtering equations with inhomogeneous sensor density}\label{sec:derivation_weighted}
We consider a latent stochastic process with an infinitesimal generator $\cL$. The most general setting in this context will be a diffusion process $X(t)\in \R^n$ for all $t \in [0,T]$,\footnote{We assume from now on that the filtering problem happens on a fixed time interval, i.e. we can anticipate data to arrive within $[0,T]$ for some $T<\infty$. This means, regardless of whether we consider the linear or nonlinear, or the discrete or continuous setting, we will have assimilated all observations at time $t=T$, and our full posterior distribution is given by $p_T$, the filtering distribution at $t=T$.} with
\begin{align}
    \dd X(t) &= b(X(t),t)\dd t + \sigma(X(t),t)\dd W_t,
\end{align}
where $b(X(t),t)\in \R^n$ is a drift coefficient, $\sigma(X(t),t)\in \R^{n\times m}$ is a matrix-valued diffusion coefficient, and $W_t\in \R^{m}$ is Brownian motion. In this case, $(\cL f)(x) = \sum_i b_i(x,t) \partial_{x_i}f(x) + \frac{1}{2}\sum_{i,j}(\sigma(x,t)\sigma(x,t)^\top)_{ij}\partial^2_{x_i,x_j}f(x)$ and its dual is $(\cL p)(x) = -\sum_i \partial_{x_i}(b_i(x,t)p(x)) + \frac{1}{2}\sum_{ij}\partial^2_{x_i,x_j}(\sigma(x,t)\sigma(x,t)^\top p(x))$. We assume that $X(t)$ is not directly accessible, but we collect a sequence of observations following a model similar to \eqref{eq:filtering_disc}, but where we can allocate a ``budget of sensor readings'' by choosing a probability vector $\alpha \in (\R^+)^N$ with $\sum_i \alpha_i = 1$ leading to observations
\begin{equation}\label{eq:filtering_disc_weighted}
    Y_{t_i} = \alpha_i g(X_{t_i}) + \sqrt{\alpha_i }\, \Gamma^{1/2}(t_i)\eps_{t_i}.
\end{equation}
Note how $\alpha_i/\sqrt{\alpha_i} = \sqrt{\alpha_i}$ is proportional to the signal-to-noise ratio (SNR) of the $i$th measurement $y_i$, and allocating the finite budget of ``precision'' values $\alpha_i$ corresponds to prioritising observations with respect to their SNR.  We can define the \textbf{empirical probability measure of observational design} $\xi$ as a weighted sum of Dirac measures
\begin{equation}
    \xi = \sum_{i=1}^N \alpha_i \delta_{t_i}(t).
\end{equation}
Letting the budget constraints be represented by the fact that $\xi$ is a probability measure is done without loss of generality, we could use any finite measure and rescale \eqref{eq:filtering_disc_weighted} accordingly.
In \eqref{eq:filtering_disc_weighted}, we have an observation mapping $H$, observational noise covariance $\Gamma(t)$, i.i.d. Gaussians $\eps_{t_i}$. 

Following the standard approach in going from discrete to continuous time in stochastic filtering, we now define a cumulative observation process as $Z(t) = \sum_{i: t_i\leq t}y_{t_i} $ which is characterised as
\begin{align*}
    Z(t) - Z(s) &= \sum_{i: s < t_i\leq t}\alpha_i g(X_{t_i}) + \sum_{i: s < t_i\leq t} \sqrt{\alpha_i}\,\Gamma^{1/2}(t_i)\eps_{t_i} = \int_s^t g(X(r))\dd \xi(r) + \Delta_\xi(t)-\Delta_\xi(s),
\end{align*}
where  $\Delta_\xi(t) :=\sum_{i: t_i\leq t} \sqrt{\alpha_i}\,\Gamma^{1/2}(t_i)\eps_{t_i}$. The object $\Delta_\xi(t)$ is a Gaussian random process such that $\Delta_\xi(s)$ and $\Delta_\xi(t) - \Delta_\xi(s)$ are independent  for $s \leq t$. The Kolmogorov extension theorem states that $\Delta_\xi(t)$ is uniquely characterised by its finite-dimensional distributions, and we can define:

\begin{defn}\label{def:Delta}
    Let $\xi$ be a probability measure on $([0,T], \mathcal B([0,T]))$ and $\Gamma:[0,T]\to \cM^+$ a time-dependent covariance matrix. Then we define $\Delta_\xi$ as the unique Gaussian process satisfying 
    \begin{align*}
    \E(\Delta_\xi(t)) &= 0\\
    \E[\Delta_\xi(r)\Delta_\xi(t)^\top]  &=  \sum_{i: t_i\leq \min\{r,t\}}\alpha_i \Gamma(t_i) = \int_0^{\min\{r,t\}} \Gamma(s)\dd \xi(s).
\end{align*}
\end{defn}
\begin{rem}

This setting encompasses two important special cases: $\xi = \sum_{i=1}^N \alpha_i \delta_{t_i}(t)$ (empirical measure, discrete probability distribution) and $\dd \xi(t) = \xi(t)\dd t$ (absolutely continuous with respect to Lebesgue measure).

We can write both of these versions notationally as $\Delta_\xi(t) = \cN(0,\int_0^t \Gamma(u)\dd\xi(u))$, where we informally write $\dd\Delta_\xi(t)$ for the infinitesimal increments.
\end{rem}

\begin{defn}[Scheduled filtering problem]
    We call inference of $X(t)$, $t\in[0,T]$ from observations $Z(t)$, where
    \begin{align*}        
    \dd X(t) &= b(X(t),t)\dd t + \sigma(X(t),t)\dd B_t\\
    \dd Z(t) &= g(x(t))\dd \xi(t) + \dd\Delta_\xi(t)
    \end{align*}
    the scheduled filtering problem, where we call the probability measure $\xi$ the \textit{sensor schedule}. 
    
    If $\xi$ has a density (notationally identified with $\xi(t)$), then by the It\=o isometry, $\Delta_\xi(t) = \int_0^t\sqrt{\xi(u)}\,\Gamma^{1/2}(u)\dd W(u)$ for standard Brownian motion $W$, so we can also write $\dd\Delta_\xi(t) = \sqrt{\xi(t)}\Gamma^{1/2}(t)\dd W(t)$. Then $Z(t)$ is a diffusion process of form
    \begin{equation}\label{eq:filtering_cts}
    \dd Z(t) = g(X(t))\xi(t)\dd t + \sqrt{\xi(t)}\,\Gamma^{1/2}(t)\dd W(t).
\end{equation}
\end{defn}

\begin{rem}
     It is tempting to view \eqref{eq:filtering_cts} as a time-rescaling of the usual measurement process in stochastic filtering,  
     \begin{align}
        \label{eq:ctsfiltobs}
         dZ(t) = g(X(t))dt + \Gamma^{1/2}(t)dW(t). 
     \end{align}
     This interpretation does not seem to hold in an obvious way, due to the dependence of $\tilde{Z}(t)$ on the signal process $\tilde{X}(t)$, as can be seen by the calculations in appendix \ref{sec:rescale}. 
\end{rem}

Whether we think of the sensor schedule $\xi$ as being a singular or absolutely continuous measure, we can write the Kushner-Stratonovich equation for the filtering problem \eqref{eq:filtering_cts} formally as
\begin{align*}
    \dd q_t(x) &= \cL^\star q_t(x)\dd t + \xi(t)^{-1}\, q_t(x)\left(g(x)\xi(t) - \bar g_t\xi(t) \right)^\top \Gamma^{-1}(t)\left(\dd Z(t) - \bar g_t\dd \xi(t)\right)\\
    &=\cL^\star q_t(x)\dd t + q_t(x)\left(g(x) - \bar g_t \right)^\top \Gamma^{-1}(t)\left(\dd Z(t) -g_t\dd \xi(t)\right).
\end{align*}
We will prefer to work with the Zakai equation instead, 
\begin{align*}
    \dd p_t(x) &= (\cL^\star p_t)(x) + p_t(x)g(x)^\top\Gamma(t)^{-1}\dd Z(t)
\end{align*}
which governs an unnormalised density $p_t(x)$ where $q_t(x) = p_t(x)/\int p_t(y)\dd y$. The logarithmic form of the Zakai equation is given by
\begin{equation}\label{eq:log-Zakai}
    \dd \log p_t(x) = \frac{\cL^\star  p_t(x)}{p_t(x)}\dd t - \frac{1}{2}\left\|g(x)\right\|_{\Gamma(t)}\dd\xi(t)+g(x)^\top \Gamma^{-1}(t)\dd  Z(t)
\end{equation}

The probability measure $\xi$ represents an \textbf{experimental design} governing the ``observational attention'' (or ``sensor schedule'') to specific subsets of the observational time horizon, either as specific points at which to measure the signal (in the discrete-time setting), or (in the continuous-time setting) as a continuous concentration of ``observational attention'' on the full interval. The fact that $\xi$ is a probability measure (and thus integrates to $1$) represents a budget limitation.

We continue with our guiding example from above.
\greybox{
Since $X(t) = x_0$ is constant, the generator is $\cL = 0$. The generalised measurement model with sensor schedule $\xi$ (assumed to have a Lebesgue density, also written as $\xi(t)$) yields an observation process
\begin{equation}\label{eq:data_guiding}
    \dd Z(t) = g(X(t),t)\xi(t)\dd t + \gamma\sqrt{\xi(t)}\dd W_t
\end{equation}
For example, if $\xi(t) = \delta_{\tau}(t)$, then this (informally) collapses to the pointwise observation problem via $y = \exp(-x_0 \cdot \tau) +\gamma \epsilon$, where $\epsilon\sim \cN(0,1)$, and where $z(s) = 0$ for $s < \tau$, and $z(s) = y$ for $s \geq \tau$. The log-Zakai equation is
\begin{equation}\label{eq:logzakai_guiding}
    \dd \log p_t(x) = \frac{1}{\gamma^2}g(x,t)\dd Z(t) - \frac{1}{2\gamma^2}g(x,t)^2\xi(t)\dd t.
\end{equation}
The next figure shows how the observation process $Z(t)$ responds to a specific choice of $\xi$, with the left column showing an exponentially distributed sensor schedule $\xi = \mathrm{Exp}(1)$, and the right column featuring a uniform $\xi = \mathrm{Unif}[2,3]$ schedule. We can see how $\dd Z$ is $0$ outside the support of $\xi$, and in fact the SNR scales with the density of $\xi$. The resulting observational data stream $Z(t)$ is then assimilated via the log-Zakai equation, with the last row in both cases showing the final filtering distribution $p_T$, compared to the prior (dashed line). We can already see that the exponential sensor schedule seems to be more informative than the specific uniform distribution chosen, as can be seen from the fact that the filtering distribution in the left column is more contracted than in the right column.

\includegraphics[width=0.5\textwidth]{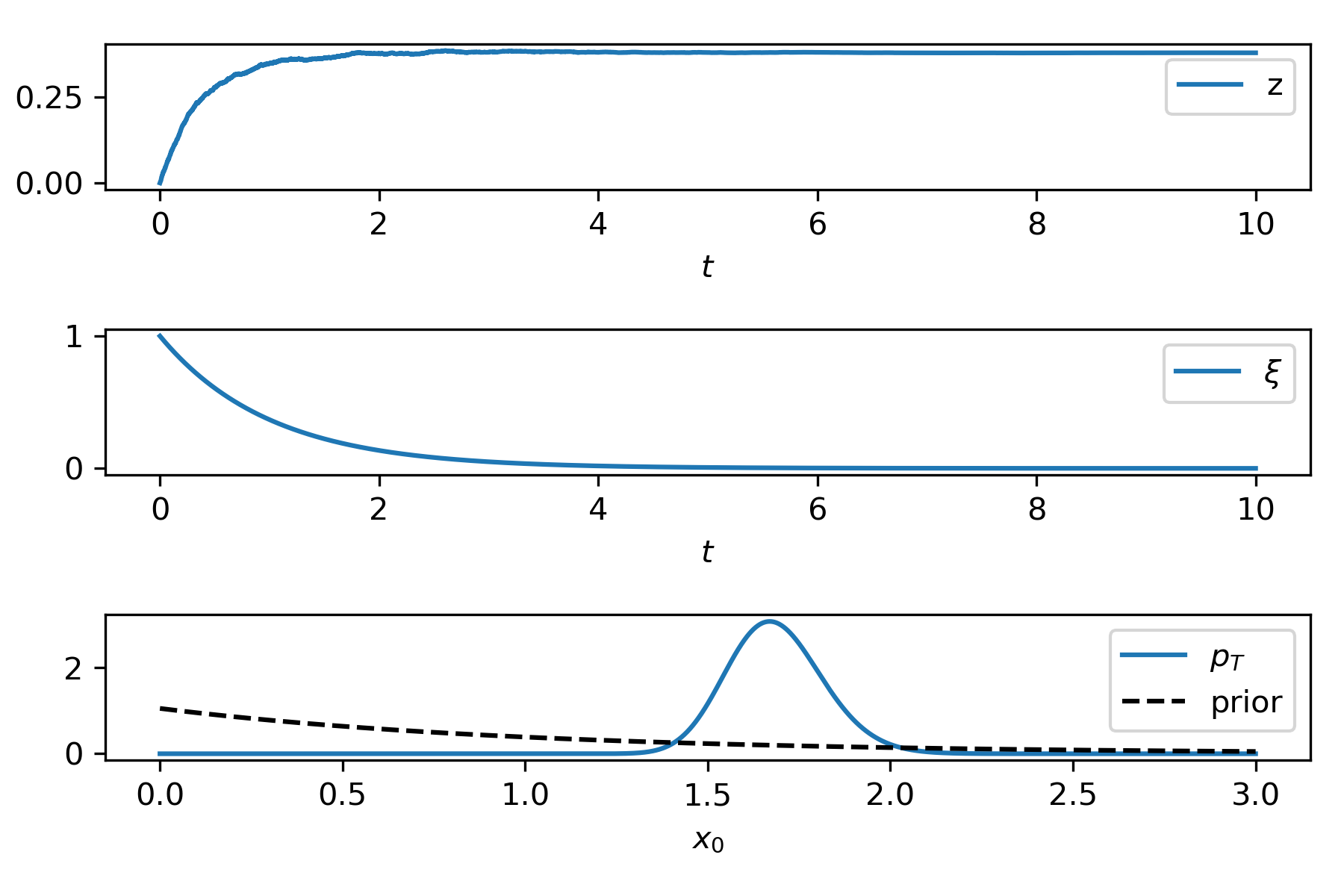}%
\includegraphics[width=0.5\textwidth]{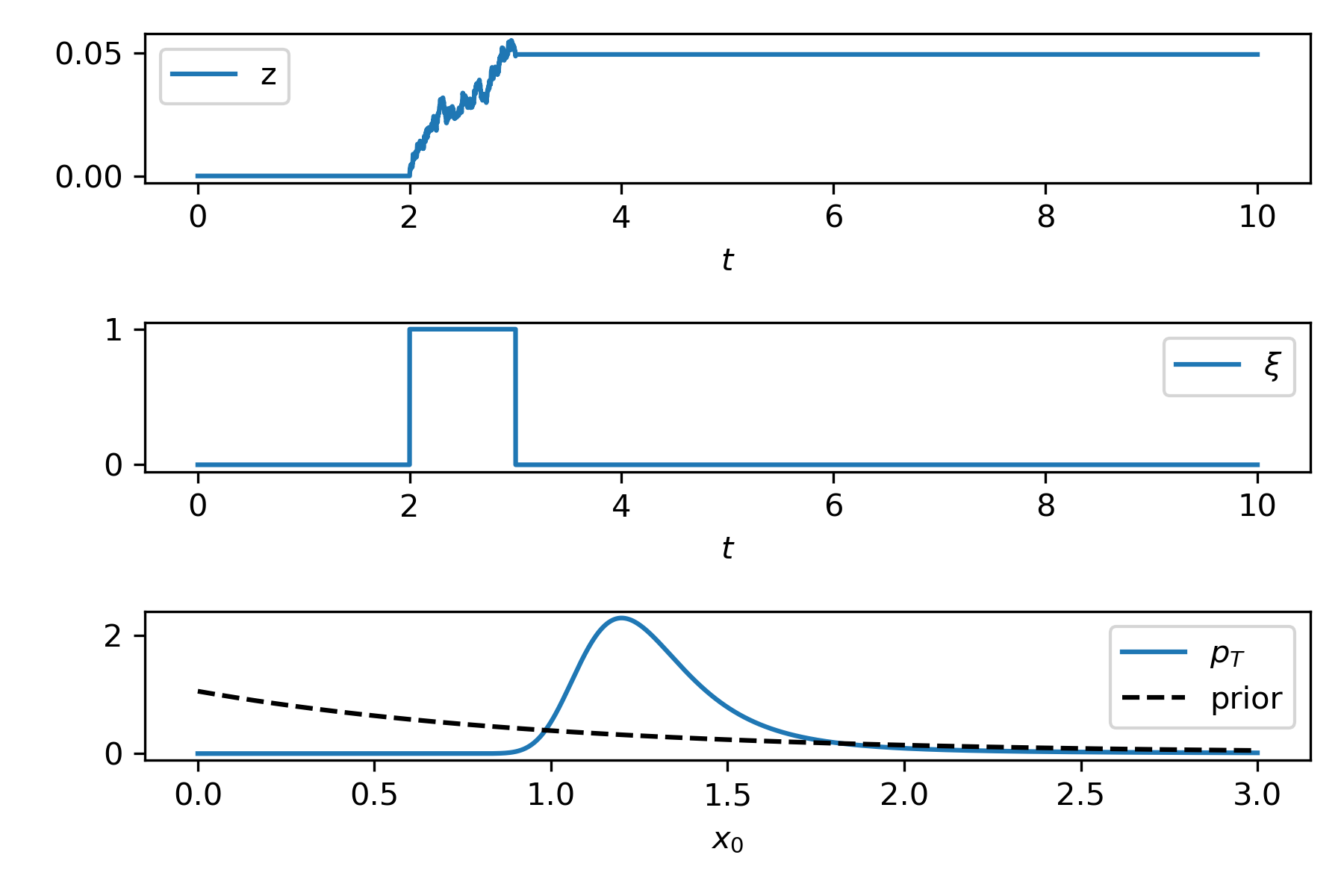}
}

\subsection{The optimal experimental design problem}\label{sec:setupOED}

We now pose an OED problem by associating a utility $U$ with desirable properties of $p_T$. This could be, for example, the KL-divergence between $p_T$ and $p_0$ (as a measure for contraction, which we typically want to be as large as possible). An alternative choice is the variance of $p_T$ (also as as a measure of contraction, where we typically want the variance to be as small as possible). Other utilities are of course possible, and inspiration may be taken from other classical measures of uncertainty known to the literature of OED \cite{pukelsheim2006optimal}.

In any case, the utility will be a function of the posterior distribution which in turn depends on the choice of sensor design $\xi$. Choosing a design leading to an optimal utility defines our optimal experimental design problem. 

To fix notation, let $\cL$ be the infinitesimal generator of the latent process $X(t)$ and $U_1 :\cP \to \R$ be a utility function to be minimised. The OED problem that we want to solve is finding the best design $\xi$ with density $\xi$ such that the expected utility is minimised:
\begin{align*}
    &\min_{\xi \in \cP} \E U_1(p),\quad \text{ where } p_t \text{ solves }\\
    \dd p_t(x) &= \cL^\star p_t(x)\dd t + p_t(x)\left(g(x) - \bar g_t \right)^\top \Gamma^{-1}(t)\left(\dd Z(t) - \bar g_t\dd \xi(t)\right),\quad \text{ and }\\
     \dd Z(t) &= g(X(t))\dd \xi(t) + \dd\Delta_\xi(t),
\end{align*}
where $\Delta_\xi$ was defined in definition \ref{def:Delta}.

\begin{rem}
    The expectation is to be taken across path realisations of $X(t)$ and $Z(t)$, similar to how EIG in standard OED contexts is computed by taking the expectation with respect to both parameters and data.
\end{rem}
\begin{rem}
    In practice we will use the log-form of the Zakai equation instead of the Kushner-Stratonovich equation.
\end{rem}

A suitable candidate for the utility function here would be to choose $U_1(p) = -\alpha D_{\mathrm{KL}}(q_T\mid\mid q_0) - \int_0^T D_{\mathrm{KL}}(q_t\mid\mid q_0)$, where $q_t$ is the normalised version of $p_t$. This optimisation problem strives to maximise a certain combination (mediated by $\alpha > 0$) of the expected information gain between the final filtering distribution $q_T$ and the prior $q_0$, and a similar integral-type term for the full filtering path. 
\greybox{
We will now convince ourselves that choosing a suitable sensor schedule $\xi$ can indeed make a difference in terms of the utility gained. Here we choose $U_1(p) = -D_{\mathrm{KL}}(q_T\mid\mid q_0),$ which measures the information gain about the unknown parameter $x_0$ after assimilating the full data $Z(t)$ governed by \eqref{eq:data_guiding}, via solving the Zakai equation \eqref{eq:logzakai_guiding}. Since we are still yet to develop the machinery for optimising $U_1$ in the next section, we just compare four different choices of $\xi$, where $\xi_i = \cN(1.5 + i, 0.5)$, for $i=0,1,2,3$. The top left figure shows these four candidate sensor schedules $\xi_i$, with the resulting observation paths shown in the top right figure. We can clearly see how the paths of $Z(t)$ have ``increased activity'' in areas where $\xi_i$ supports probability mass. If we solve the log-Zakai equations for each of these observational paths, we get the corresponding final-time filtering distributions $p_T$ (lower left figure). We can see that setting $0$ and $3$ (blue and red) have relatively little posterior contraction (as compared to the prior distribution shown with a dashed line). Settings $1$ and $2$ (yellow and green) seem to have collected more information from the data assimilation procedure. This is summarised in a plot of the value utility function for these four settings, which shows that these settings indeed lead to the largest KL-divergence among the one compared.
\includegraphics[width=0.5\linewidth]{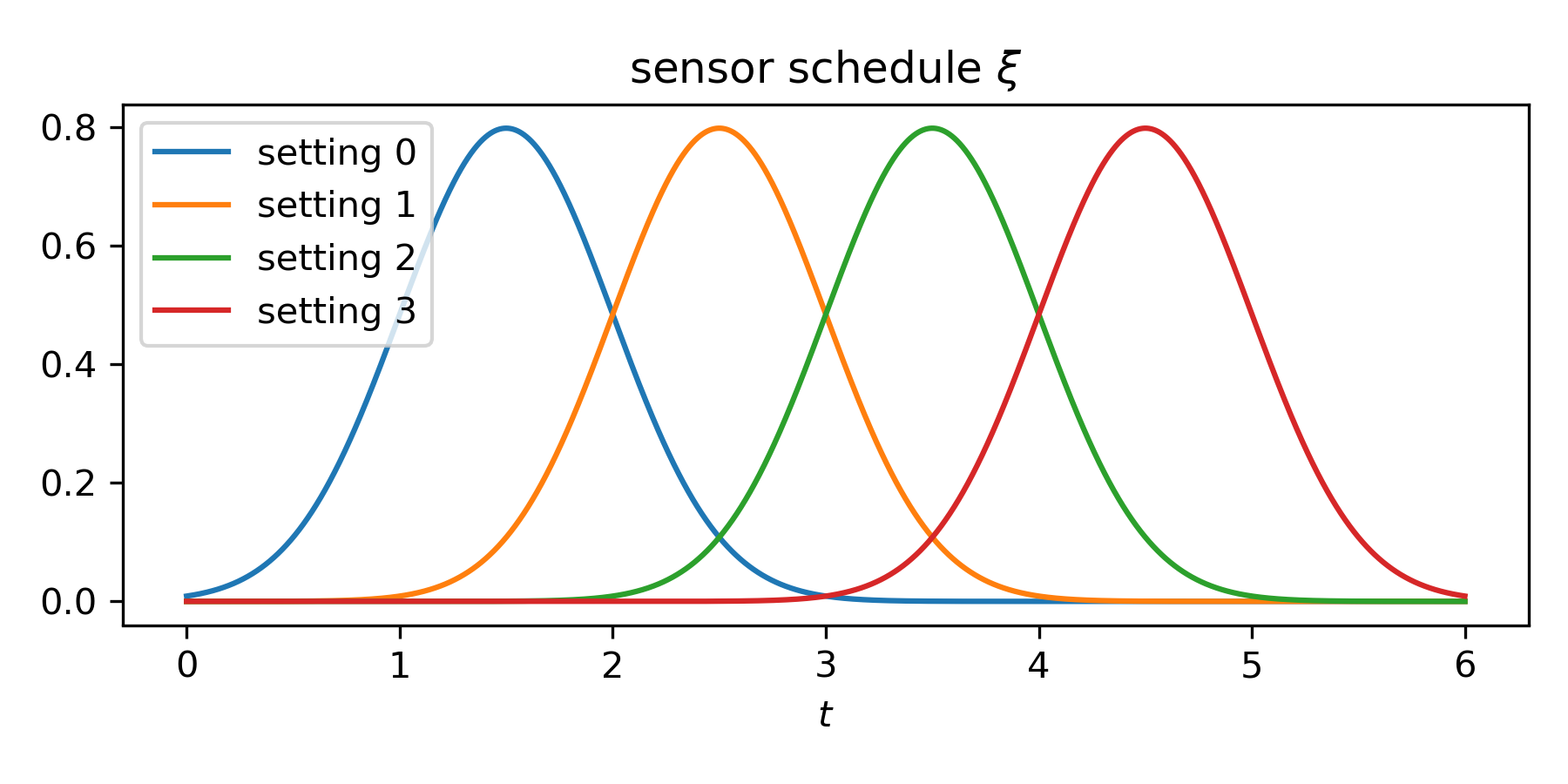}%
\includegraphics[width=0.5\linewidth]{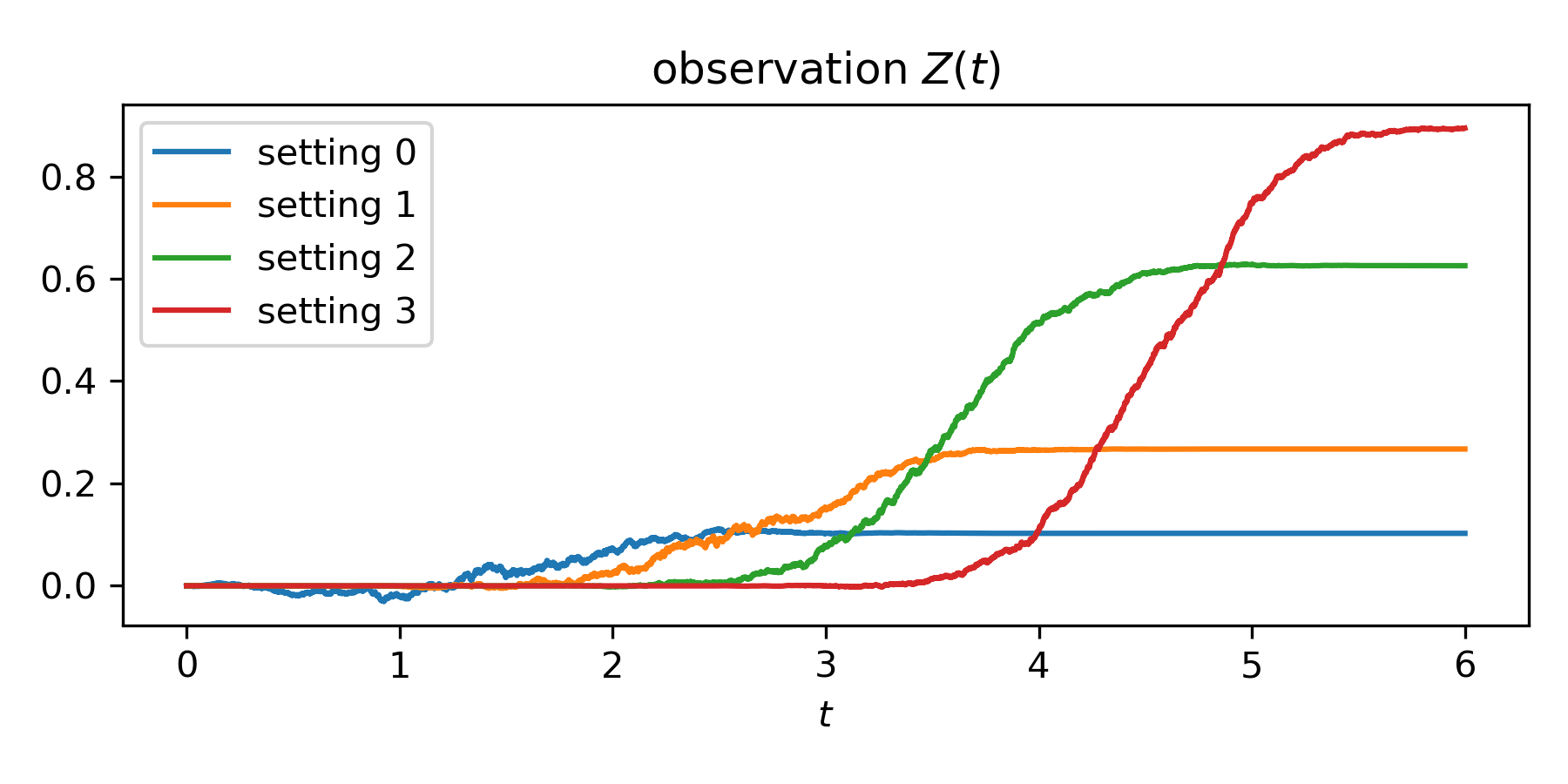}\\
\includegraphics[width=0.5\linewidth]{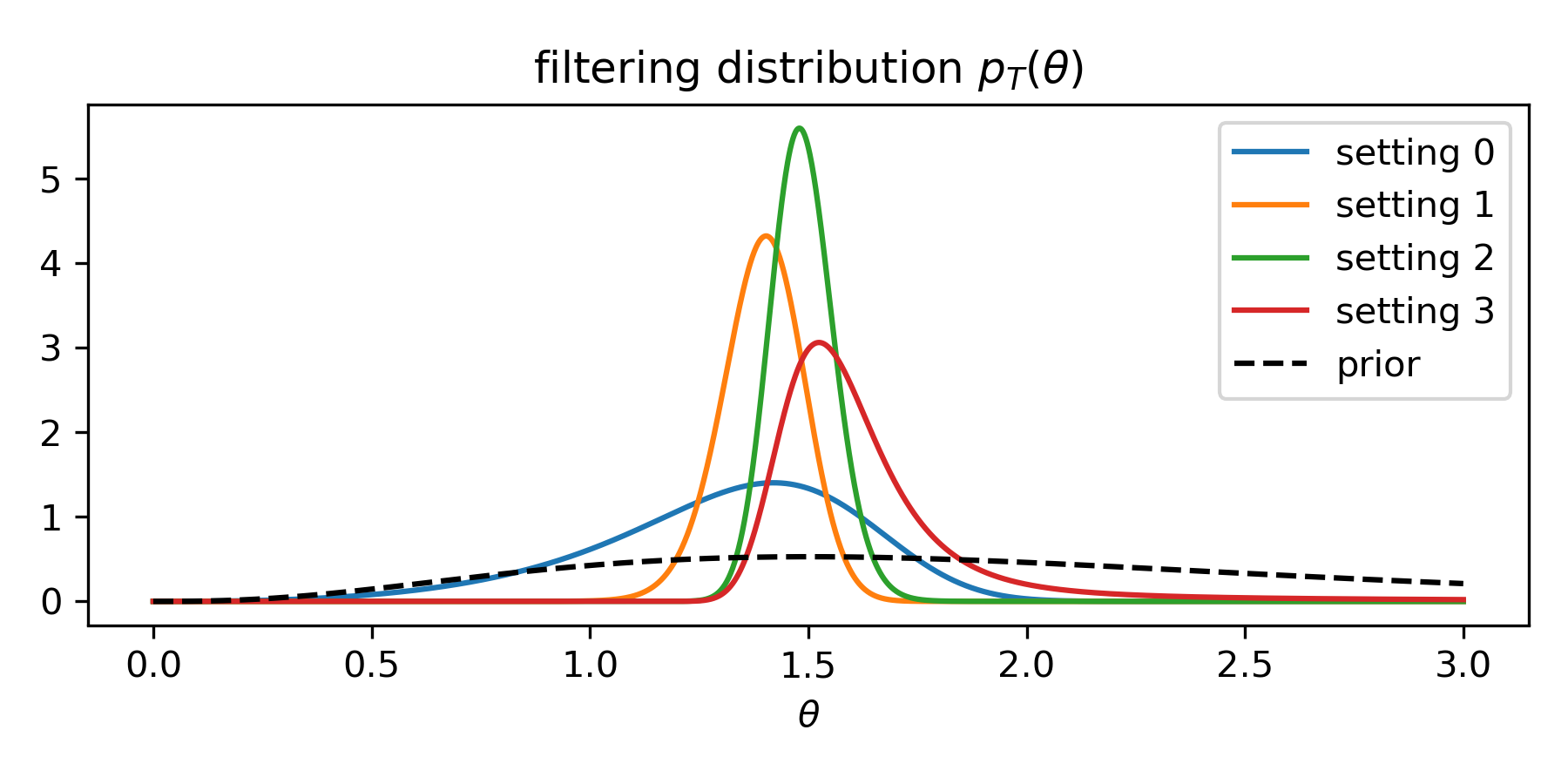}%
\includegraphics[width=0.5\linewidth]{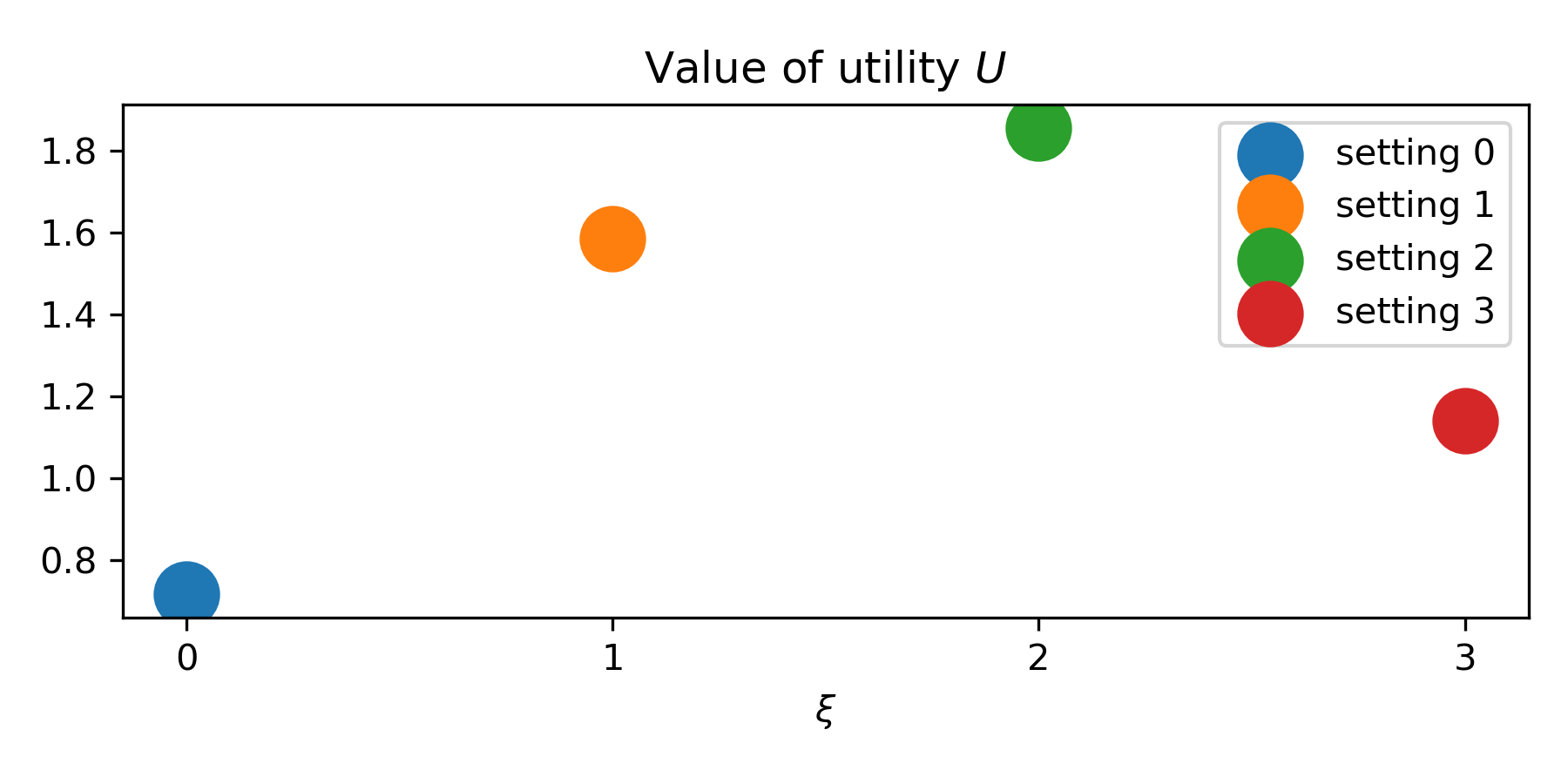}
We note that the specific outcomes for $Z(t),p_T,$ and the value of the utility are path-dependent, and the utility would strictly speaking need to be averaged over many such runs so we get a better (lower variance) estimate for $\E U_1(p)$, but we forgo this for the sake of simplicity.
}

We now consider the special case where the system dynamics and the observation operator are linear, and all errors of stochasticity are Gaussian. This yields much simpler weighted Kalman-Bucy equations.
\subsection{Linear Gaussian case}
If the signal has linear and Gaussian dynamics, and the observations are of the form \eqref{eq:filtering_cts}, with $g(x)=Hx$ for a matrix $H$,
\begin{equation}\label{eq:linear_Gauss_weighted}
\begin{split}
    \dd X(t) &= L(t) X(t)\dd t + \Sigma^{1/2}(t)\dd B_t\\
    \dd Z(t) &= \xi(t)H X(t)\dd t + \sqrt{\xi(t)} \Gamma^{1/2}(t)\dd W_t,
\end{split}
\end{equation}
for $t\in [0,T]$, then $X(t)|\{Z(s)\}_{s\leq t}$ follows a Gaussian distribution $\cN(m(t), C(t))$ for all $t$, with corresponding mean $m(t)$ and covariance $C(t)$ evolving according to the Kalman-Bucy equations
\begin{equation}\label{eq:Kalman_Bucy_weighted}
\begin{split}
    \dd m(t) &= L(t) m(t) \dd t+ C(t) H^\top\Gamma(t)^{-1}\left(\dd Z(t)- Hm(t)\dd \xi(t) \right)\\
    \frac{\dd C_\xi(t)}{\dd t} &=L(t) C_\xi(t) + C_\xi(t) L(t)^\top + \Sigma(t) - \xi(t) C_\xi(t) H^\top \Gamma^{-1}(t) H C_\xi(t)
\end{split}
\end{equation}

We now proceed in setting up the OED problem of optimal sensor scheduling. For simplicity, we consider the case where the utility is a function of the filtering covariance alone, and where $\xi$ is absolutely continuous with density, in which case we want to solve
\begin{align*}
    &\min_{\xi \in \cP} U_2(C_\xi)
\end{align*}

From now on we will suppress the dependence of $C_\xi$ on $\xi$ notationally.

We want to note here that in this specific case, the utility is a deterministic quantity in terms of $\xi$, independent of the actual observational data (since the Riccati equation governing $C$ does not react to neither the actual underlying path of the latent state $X(t)$, nor the observation $\dd z(t)$), so there is no need to take the expectation over possible realisations of this data.

A suitable utility function here arises from $U_2(C) = \tr(C(T))$, so the optimisation problem is aiming towards minimising the variance of the Gaussian filtering distribution at the final time $T$.

\section{Solving the OED problem using adjoint derivatives}\label{sec:solvingOED}
We have now formulated the OED problem of finding the best sensor schedule $\xi$. What remains to be done is to derive an expression for the gradient of the utility function with respect to $\xi$, so we can set up efficient numerical methods. We will do this now, with the help of the machinery of ``adjoints'', well-known in the optimal control community \cite{giles2002adjoint,giles2006smoking}. A recapitulation of this concept can be found in appendix \ref{sec:recap_adjoint}, and now we will begin with the linear and Gaussian case, and then derive the gradient in the more general nonlinear case afterwards. 
\subsection{Linear Gaussian case}
\begin{lem}\label{lem:gradient_matrix}
We consider a utility $U_2(C) = U_\mathrm{final}(C(T)) + \int_0^T U_\mathrm{int}(C(t))\dd t$ for some smooth functions $U_\mathrm{final},U_\mathrm{int}$, where 
\begin{align}\label{eq:C_ODE_adjoint}
    \dot C(t) = f(C(t), \xi(t), t),\qquad C(0) = C_0
\end{align}
is a matrix-valued ODE. Then the gradient of $U_2$ with respect to $\xi$ can be evaluated via $\DD_\xi U_2(\xi) = \eta$ where
\begin{align*}
    \eta(t) = (\partial_\xi f(C(t),\xi(t),t))^\star \Lambda(t)
\end{align*}
and
\begin{align*}
    -\dot \Lambda(t) &= (\partial_C f(C(t),\xi(t),t))^\star \Lambda(t) + \partial_C U_\mathrm{int}(t,C(t)), &\Lambda(T) &= \partial_C U_\mathrm{final}(C(T))
\end{align*}
\end{lem}
\begin{proof}
We set $H(\dot C, C, \xi, t) = \dot C(t) - f(C(t), \xi(t), t)$ (which means that $H = 0$ characterises solutions $C$ of \eqref{eq:C_ODE_adjoint}). This defines a Lagrangian 
\begin{align*}
    \mathbf L(\dot C, C, \xi, \Lambda, \Omega) &= U_\mathrm{final}(C(T)) + \int_0^T U_\mathrm{int}(t,C(t))\dd t \\
    &\quad- \int_0^T \left\llangle \Lambda(t), H(\dot C, C, \xi, t)\right\rrangle \dd t - \llangle \Omega, C(0)-C_0\rrangle - \nu\cdot \left(\int_0^T \xi(t)\dd t - 1\right)- \int_0^T \mu(t)\xi(t)\dd t ,
\end{align*} for a non-negative Lagrange multiplier $\mu$, 

where $\Lambda\in \cC([0,T],\cM)$ and $\Omega\in\cM$ are matrix-valued Lagrangian multipliers and $\nu$ is a real-valued Lagrangian multiplier. Note that the inner product here is $\llangle M,N\rrangle = \sum_{ij}M_{ij}N_{ij} = \tr(M^\top N)$. Now we can compute the Fr\'echet derivative of the Lagrangian with respect to the design $\xi$ as (we write $H(\dot C, C, \xi, t) = H$ for convenience)
\begin{align*}
    \DD_\xi &\mathbf L(\dot C, C, \xi, \Lambda, \Omega)[h] = \llangle\DD_C U_\mathrm{final}(C(T)), \DD_\xi C(T)[h]\rrangle +\int_0^T \llangle\DD_C U_\mathrm{int}(t,C(t)), \DD_\xi C(t)[h]\rrangle \ddt  \\
    &- \int_0^T \left\llangle \Lambda(t), \partial_{\dot C}H \left[\DD_\xi \dot C(t)[h]\right] + \partial_C H \left[\DD_\xi C(t)[h]\right] + \partial_\xi H[h]\right\rrangle\dd t \\&- \langle \Omega, \DD_\xi C(0)[h]\rangle - \nu \int_0^T h(t)\dd t - \int_0^T \mu(t)h(t)\dd t
    \intertext{Now we use the fact that $\partial_{\dot C}H = I$, and $\partial_CH = -\partial_C f,$ and $\partial_\xi H = -\partial_\xi f$:}
    &=  \llangle\DD_C U_\mathrm{final}(C(T)), \DD_\xi C(T)[h]\rrangle +\int_0^T \llangle\DD_C U_\mathrm{int}(t,C(t)), \DD_\xi C(t)[h]\rrangle \ddt  - \int_0^T \llangle \Lambda(t), \DD_\xi \frac{\dd}{\dd t} C(t)[h]\rrangle\dd t \\
    &+ \int_0^T \left\llangle \Lambda(t),  \partial_C f\left[ \DD_\xi C(t)[h]\right] + \partial_\xi f[h]\right\rrangle\dd t - \llangle \Omega, \DD_\xi C(0)[h]\rrangle- \int_0^T h(t)(\nu + \mu(t))\dd t
    \intertext{Integration by parts (with respect to $t$) on the second integral, and using $\DD_\xi C(0) = 0$ yields}
    &= \llangle\DD_C U_\mathrm{final}(C(T)), \DD_\xi C(T)[h]\rrangle +\int_0^T \llangle\DD_C U_\mathrm{int}(t,C(t)), \DD_\xi C(t)[h]\rrangle \ddt  - \left[\llangle\Lambda(t), \DD_\xi C(t)[h]\rrangle \right]_0^T\\
    &+\int_0^T \llangle \dot\Lambda(t), \DD_\xi C(t)[h]\rrangle + \left\llangle \Lambda(t),  \partial_C f\left[ \DD_\xi C(t)[h] \right]+ \partial_\xi f[h]\right\rrangle\dd t-  \int_0^T (\nu+\mu(t)) h(t)\dd t \\
    &= \llangle  \DD_C U_\mathrm{final}(C(T))-\Lambda(T), \DD_\xi C(T)[h]\rrangle + \int_0^T \llangle \dot\Lambda(t) + [\partial_C f]^\star \Lambda +\DD_C U_\mathrm{int}(t,C(t)), \DD_\xi C(t)[h]\rrangle \dd t \\
    &+\int_0^T \llangle \Lambda(t), \partial_\xi f[h]\rrangle \dd t-  \int_0^T (\nu+\mu(t)) h(t)\dd t 
\end{align*}
Now we set $\Lambda(T) = \DD_C U_\mathrm{final}(C(T))$, $-\dot \Lambda(t) = (\partial_C f)^\star \Lambda(t) + \DD_C U_\mathrm{int}(t,C(t)),$ and use the fact that $\partial_\xi f: T_\xi \cP\to \cM$, so $(\partial_\xi f)^\star: \cM^\star \to T_\xi^\star \cP$, and thus (since we can identify the Euclidean space $\cM$ with its dual)
\[
\int_0^T \llangle \Lambda(t), \partial_\xi f[h]\rrangle \dd t = \int_0^T  (\partial_\xi f)^\star[\Lambda(t)]\, h(t) \dd t
\]
which simplifies the Lagrangian derivative to $ \DD_\xi \mathbf L(\dot C, C, \xi, \Lambda, \Omega)[h] =   \int_0^T \left\{ (\partial_\xi f)^\star[\Lambda(t)] - \nu-\mu(t)\right\}\, h(t) \dd t$, i.e.,
\begin{align*}
    \DD_\xi &\mathbf L(\dot C, C, \xi, \Lambda, \Omega)=  (\partial_\xi f)^\star[\Lambda(t)] - \nu - \mu(t)
\end{align*}    
where $\nu,\mu(t)$ are chosen to satisfy the constraints on $\xi$ to be a valid probability density.

\end{proof}

\begin{lem}\label{lem:grad_lin_gauss}
    We consider the linear, Gaussian setup \eqref{eq:linear_Gauss_weighted} with Kalman-Bucy equations \eqref{eq:Kalman_Bucy_weighted} and set $U:\cM \to \R$ to be a utility function. We abbreviate $A(t) := H^\top\Gamma(t)^{-1}H$.  Then the gradient of $V(\xi):= U_\mathrm{final}(C(T)) + \int_0^T U_\mathrm{int}(C(t))\dd t$ is given by
    \begin{align*}
        \delta_\xi V(\xi)(t) &=-\llangle \Lambda(t), C(t)A(t)C(t)\rrangle
    \end{align*}
    and where
    \begin{align*}
    -\dot \Lambda(t) &= L^\top \Lambda(t) + \Lambda(t) L - \xi A(t)C(t)\Lambda(t) - \xi \Lambda C(t)A(t)+ \partial_C U_\mathrm{int}(t,C(t))\\
    \Lambda(T) &= \partial_C U_\mathrm{final}(C(T)).
\end{align*}
In particular:
\begin{itemize}
    \item If $U(M) = \tr(M)$, then $\Lambda(T)= I$.
    \item If $U(M) = \log\det(M)$, then $\Lambda(T) = 2C^{-1}(T) - C^{-1}(T)\circ I$.
\end{itemize}
\end{lem}
\begin{proof}
In this proof, we will sometimes drop dependency on $t$ notationally if convenient. Since $f(\dot C, C, \xi, t) = LC + CL^\top + \Sigma - \xi(t) CAC$, we have
\begin{align*}
    \frac{\partial f_{kl}}{\partial C_{ij}} &= L_{ki}\delta_{lj} + \delta_{ik}L_{lj} - \xi(t)\delta_{lj}(CA)_{ki} - \xi(t)\delta_{ki}(AC)_{jl}\\
    \frac{\partial f_{kl}}{\partial\xi} &= (-C(t)AC(t))_{kl}
\end{align*}
Then from lemma \ref{lem:gradient_matrix},
\begin{align*}
    \llangle \Lambda, \partial_C f M\rrangle &= \sum_{kl} \Lambda_{kl} \sum_{ij}\frac{\partial f_{kl}}{\partial C_{ij}} M_{ij} = \sum_{ij}M_{ij}\underbrace{\sum_{kl}  \frac{\partial f_{kl}}{\partial C_{ij}}\Lambda_{kl}}_{[(\partial_C f)^\star(\Lambda)]_{ij}} = \llangle (\partial_C f)^\star(\Lambda), M\rrangle,
\end{align*}
which leads to 
\begin{align*}
    -\dot \Lambda_{ij}(t) &= \sum_{kl}  \frac{\partial f_{kl}}{\partial C_{ij}}\Lambda_{kl} + (\partial_C U_\mathrm{int}(t,C(t)))_{ij}= \sum_{kl}\left[L_{ki}\delta_{lj} + \delta_{ik}L_{lj} - \xi\delta_{lj}(CA)_{ki} - \xi\delta_{ki}(AC)_{jl}\right]\Lambda_{kl} + (\partial_C U_\mathrm{int}(t,C(t)))_{ij}\\
    &= \sum_k L^\top_{ik}\Lambda_{kj} + \sum_l \Lambda_{il}L_{lj} - \sum_k (CA)^\top_{ik}\Lambda_{kj}- \sum_l \Lambda_{il}(AC)^\top_{lj} + (\partial_C U_\mathrm{int}(t,C(t)))_{ij}\\
    &= \left[L^\top \Lambda +\Lambda L - \xi(CA)^\top \Lambda - \xi\Lambda(AC)^\top\right]_{ij} + (\partial_C U_\mathrm{int}(t,C(t)))_{ij}
\end{align*}
Which means that (using $A=A^\top,C=C^\top$),
\begin{align*}
    -\dot \Lambda = L^\top \Lambda + \Lambda L - \xi AC\Lambda - \xi \Lambda CA + \partial_C U_\mathrm{int}(t,C(t))
\end{align*}

The gradient of $V(\xi)$ is then given by $\DD_\xi V(\xi)[h] = \int_0^T \llangle \Lambda(t), C(t)A(t)C(t)\rrangle  h(t)\dd t,$ so 
\begin{equation*}
    \DD_\xi V(\xi)(t) =  -\llangle \Lambda(t), C(t)A(t)C(t)\rrangle,
\end{equation*}

 where $\Lambda_T$ depends on the specific utility, see lemma \ref{lem:matrix_derivatives}.

\end{proof}
\subsection{Nonlinear case}
For simplicity, we are going to focus on the case where the hidden signal $X(t)$ is a diffusion process. The results in this sections can also be extended to the case where $X(t)$ is driven by a more general infinitesimal generator.
\begin{lem}\label{sec:main_lemma}
    We assume that the hidden signal is a stochastic process $\{X(t)\}_{t\in[0,T]}$ of form
    \begin{align*}
        \dd X(t) = f(X(t))\dd t + \Sigma^{1/2} \dd B_t 
    \end{align*}
    with infinitesimal generator $(\cL p)(x) = f(x) \cdot \nabla p(x) + \frac{1}{2}\tr[\Sigma \nabla^2 p(x)]$, and that data is generated via an observation mapping $g: X\to Y$ and a probability measure $\xi\in\cP$ (the sensor schedule with density $\xi(t)$) as
\begin{equation}
    \dd Z(t) = g(X(t))\xi(t)\dd t + \sqrt{\xi(t)}\,\Gamma^{1/2}(t)\dd W_t.
\end{equation}

We further assume that we have a utility which is a function of the posterior distribution after assimilation of all the data, i.e.
\begin{equation}
    V(\xi) = \E U(p),
\end{equation}
where $p_t$ is the filtering distribution of $X(t)|\{Z(s)\}_{s=0}^t$ given by the log-Zakai equation
\begin{align*}
    \dd \log p_t(x) &= \frac{\cL^\star  p_t(x)}{p_t(x)}\dd t - \frac{1}{2}\left\|g(x)\right\|_{\Gamma(t)}\dd\xi(t)+g(x)^\top \Gamma^{-1}(t)\dd  Z(t)
\end{align*}
and the path-dependent utility is a generic integral operator of both final filtering distribution and its path,
\begin{align*}
    U(p) = U_\mathrm{final}(p_T) + \int_0^T U_\mathrm{int}(p_s)\dd s = \int u_\mathrm{final}(x,p_T(x))\dd x +  \int_0^T \int u_\mathrm{int}(x,(p_s(x))\dd x.  
\end{align*}
Then the Fr\'echet derivative of $V(\xi)= \E U(p)$ with respect to $\xi$ is given by $\DD_\xi V(\xi) = \eta,$ with 
    \begin{equation} \label{eq:frechet_der}
        \eta(t) = \E\int   \lambda_t(x)  \left( - g(x)^\top \Gamma^{-1} g(X(t))    +\frac{1}{2} \| g(x)\|^2_{\Gamma^{-1}} \right) \dd x
    \end{equation}

    and where (pointwise for a specific path realisation leading to a filtering distribution $p_t$)
\begin{align}\label{eq:bwd_KS}
    -\frac{\dd\lambda_t(x)}{\dd t} &= -p_t(x) D_{p_t}u_\mathrm{int}(x,p_t(x)) -\nabla \cdot \left(\lambda_t(x) [f(x) -\Sigma \nabla \log p_t(x)] + \frac{1}{2} \Sigma \nabla \lambda_t(x) \right)\\
    \lambda_T(x) &= p_T(x) D_{p_T}u_\mathrm{final}(x,p_T(x))
\end{align}

\end{lem}
\begin{rem}
    If we set $u_\mathrm{final}(x,p_s(x)) = q_s(x)\log \frac{q_s(x)}{q_0(x)}$, then $U(p) = D_{\mathrm{KL}}(q_T\mid\mid q_0)$ (recall that $q_t$ is the normalised version of $p_t$).
\end{rem}
\begin{rem}
    Using the log-Zakai (instead of the more common Zakai equation, or even the Kushner-Stratonovich equation) has two main benefits: First, we found the log-Zakai equation to be more suitable in our context since we don't need to numerically enforce non-negativity of $p_t$ by working with the log-density directly. Second, if we use the (non-log) Zakai equation, then the (backwards) adjoint equation becomes a stochastically driven semilinear parabolic PDE (like the one for $p_t$ itself) of form
    \begin{equation*}
        \text{adjoint from non-log Zakai:} \quad-\dd \lambda_t(x) = \cL \lambda_t(x)\dd t + p_t(x)\lambda_t(x) g(x)^\top \Gamma(t)^{-1}\dd Z(t),
    \end{equation*}
    while \eqref{eq:bwd_KS} leads to a smooth solution since there is no direct influence of the stochastic forcing $\dd Z(t)$, only via its influence on the filtering distribution $p_t$.
\end{rem}
\begin{proof}
The Lagrangian\footnote{We drop Lagrangian multipliers for nonnegativity and normalisation for simplicity, keeping in mind that the gradient has to be chosen to respect these constraints.} is then (where in the below, $p$ means the entire trajectory $\{p_t\}_{t \in [0,T]}$),
\begin{align*}
    \mathbf L_T(p) = \E U(p) - \E \int_0^T \int \lambda_t(x) \left(\dd  \log p_t(x) - \frac{\cL^\ast p_t(x)}{p_t(x)} \dd t -  g(x)^\top \Gamma^{-1} \dd Z(t) +\frac{1}{2}\xi(t) \| g\|^2_{\Gamma^{-1}}\dd t \right) \dd x
\end{align*}
This means that
\begin{align*}
    D_\xi \mathbf L(p)[h] &=  \E D_p U(p) [D_\xi p [h]] - \E \int_0^T \int \lambda_t(x) \left( D_\xi (\dd \log p_t(x))[h]  - \frac{\cL^\ast D_\xi p_t(x)[h]}{p_t(x)} + \frac{\cL^\ast p_t(x)}{p_t(x)^{2}} D_\xi p_t(x) \dd t \right)\dd x \\
    & - \E \int_0^T\int \lambda_t(x)  \left( - g(x)^\top \Gamma^{-1} \DD_\xi\dd Z(t)  +\frac{1}{2} h(t)\| g(x)\|^2_{\Gamma^{-1}}dt \right) \dd x \\
     \intertext{We perform stochastic integration by parts with respect to time on the integral $\int_0^T \int \lambda_t(x) D_\xi (\dd\log p_t(x))[h]$ (yielding an additional It\=o correction term containing the quadratic covariation between $\lambda_t$ and $D_\xi\log p_t$), and compute the Fr\'echet derivative of the observation path with respect to $\xi$, leading to}
     &= \E D_p U(p) [D_\xi p [h]] - \E \int_0^T \int \lambda_t(x) \left(   - \frac{\cL^\ast D_\xi p_t(x)[h]}{p_t(x)} + \frac{\cL^\ast p_t(x)}{p_t(x)^{2}} D_\xi p_t(x) \dd t \right)\dd x \\
    & - \E \int_0^T\int\lambda_t(x)  \left( - g(x)^\top \Gamma^{-1} \left(g(X(t))h(t)\dd t + \frac{1}{2}h(t) \xi(t)^{-1/2} \Gamma^{1/2} \dd W_t \right)  +\frac{1}{2} h(t)\| g(x)\|^2_{\Gamma^{-1}}dt \right) \dd x \\
    & -\E \int \left( \lambda_T D_\xi  \log p_T[h] - \int_0^T \dd\lambda_t D_\xi \log p_t(x)[h] dt -\frac{1}{2} d\lambda_t d(D_\xi \log p_t[h])   \right) \dd x\\
    \intertext{We now make the ansatz $d\lambda_t = \phi(x, \xi(t)) dW_t + \psi(x, \xi(t))dt $ with $\phi$ and $\psi$ yet to be determined. At the same time, we know that (from the Zakai equation) $d D_\xi \log p_t [h]= \text{[drift terms]}+ g(x)^\top \Gamma^{-1} \frac{h(t)}{2\xi^{1/2}(t)} \Gamma^{1/2} \dd W_t$, so that we get the following general form of the quadratic covariation term: $d\lambda_t d(D_\xi \log p_t[h]) = \frac{g(x)^\top \Gamma^{-1}(t)\phi(x,\xi(t))h(t) }{2\xi^{1/2}(t)}\dd t$. In addition, we recall that $U(p) = U_\mathrm{final}(p_T) + \int_0^T U_\mathrm{int}(p_s)\dd s = \int u_\mathrm{final}(x,p_T(x)))\dd x +  \int_0^T \int u_\mathrm{int}(x,(p_s(x)))\dd x$, and use the adjoint operator to simplify the expression $-\int \lambda_t(x) \frac{\cL^\star D_\xi p_t(x)[h]}{p_t(x)}\dd t = -\int \cL\left(\frac{\lambda_t}{p_t} \right)(x)D_\xi p_t(x)[h]\dd t$, so}
    &= \E \int D_{p_T}u_\mathrm{final}(x,p_T(x))  D_\xi p_T [h](x)\dd x + \E \int_0^T\int D_{p_t}u_\mathrm{int}(x,p_t(x))[D_\xi p_t[h](x)]\dd x\dd t\\
    &- \E \int_0^T \int \left[ -\mathcal{L} \left( \frac{\lambda_t}{p_t} \right)(x)   
+\frac{\lambda_t(x)}{p_t(x)} \left(     \frac{\cL^\ast p_t(x)}{p_t(x)}  \right)\right]D_\xi p_t(x)  \dd t \\
    & - \E \int_0^T\int\lambda_t(x)  \left( - g(x)^\top \Gamma^{-1} \left(g(X(t))h(t)\dd t + \frac{1}{2}h(t) \xi(t)^{-1/2} \Gamma^{1/2} \dd W_t \right)  +\frac{1}{2} h(t)\| g(x)\|^2_{\Gamma^{-1}}dt \right) \dd x \\
    & - \E \int\left( \lambda_T \frac{D_\xi   p_T[h]}{p_T(x)} - \int_0^T \dd\lambda_t(x) \frac{D_\xi p_t(x)[h]}{p_t(x)}  -\frac{1}{4} \int_0^T \frac{h(t)}{\xi^{1/2} }\Gamma^{1/2} \phi dt   \right) \dd x\\
    &= \E \int \left(D_{p_T}u_\mathrm{final}(x,p_T(x)) - \frac{\lambda_T(x)}{p_T(x)}\right)  D_\xi p_T [h](x)\dd x\\
    &+\E \int_0^T \int D_\xi p_t[h](x) \left(D_{p_t}u_\mathrm{int}(x,p_t(x))\ddt + \cL \left(\frac{\lambda_t}{p_t} \right)(x)\ddt - \frac{\lambda_t(x)}{p_t(x)}\frac{\cL^\star p_t(x)}{p_t(x)}\ddt+ \frac{\dd \lambda_t(x)}{p_t(x)}\right)\dd x\\
    &+ \E \int_0^Th(t)\int \lambda_t(x)  \left( -\langle g(x), g(X(t))\rangle_{\Gamma(t)} + \frac{1}{2}\|g(x)\|_\Gamma(t)^2 \right)   + \frac{1}{4\xi^{1/2}(t)}\Gamma^{1/2}(t)\phi(x,\xi)\dd t \dd x\\
    &+\E \int_0^T\int \lambda_t(x)h(t)\left( - \frac{1}{2}g(x)^\top \Gamma^{-1/2} \xi(t)^{-1/2}  \dd W_t + \frac{1}{4\xi^{1/2}(t)}\Gamma^{1/2}(t)\phi(x,\xi(t))\right)
\end{align*}
The four lines in the last expression determine the following components of interest:

By setting $\lambda_T(x) = p_T(x) D_{p_T}u_\mathrm{final}(x,p_T(x))$, the first line vanishes. The second line can be dropped by posing the (backwards) adjoint differential equation
\begin{align*}
    -\partial_t\lambda_t(x) = -\lambda_t(x) \frac{(\cL^\star p_t)(x)}{p_t(x)} + \cL\left( \frac{\lambda_t(x)}{p_t(x)}\right)(x) + D_{p_t}u_\mathrm{int}(x,p_t(x))
\end{align*}
With this choice, we see that the diffusion coefficient of the dynamics of $\lambda$ is $0$, so $\phi = 0$, and the term stemming from the quadratic covariation vanishes. Since we take the expectation with respect to $W_t$,  the stochastic integral is dropped, as well (under mild integrability assumptions on the integrand, making sure the local martingale given by the It\=o integral is actually a martingale). This means that the fourth line vanishes. The third line, finally, evaluates to an expression yielding the Fr\'echet derivative of the utility functional with respect to $\xi$:

\[\mathbb{E}D_\xi \mathbf L(p) = \E \int   \lambda_t(x)  \left( - g(x)^\top \Gamma^{-1} g(X(t))    +\frac{1}{2} \| g(x)\|^2_{\Gamma^{-1}} \right) \dd x\]

 We now put in a bit more effort to simplify the first two terms in the adjoint equation, using the special structure of $\cL$ in our case where 
\begin{align*}
    \cL^\star p_t(x) &= -\nabla \cdot (p_t(x) f(x)) + \frac{1}{2} \tr[\Sigma \nabla^2 p_t(x)] \\
    \cL p_t(x)  & = f(x) \cdot \nabla p_t(x) + \frac{1}{2}\tr[\Sigma \nabla^2 p_t(x)]
\end{align*}
We can then define $\mu_t = \frac{\lambda_t}{p_t}$ and evaluate the term in the right hand side of the adjoint equation as
\begin{align*}
    - p_t(x)\cL \left(\frac{\lambda_t}{p_t} \right)(x) + \lambda_t(x)\frac{\cL^\star p_t(x)}{p_t(x)} &= -p_t \left(  f \cdot \nabla \left(\frac{\lambda_t}{p_t} \right) + \frac{1}{2}\tr\left[\Sigma \nabla^2 \left( \frac{\lambda_t}{p_t} \right) \right] \right) \\
    &+  \frac{\lambda_t}{p_t} \left( -\nabla \cdot (p_t(x) f(x)) + \frac{1}{2} \tr[\Sigma \nabla^2 p_t(x)] \right)  \\
    & = -\nabla \cdot \left( \lambda_t f \right) + \frac{1}{2} \tr \left[ \frac{\lambda_t}{p_t} \nabla^2 p_t \Sigma - p_t \nabla^2 \left( \frac{\lambda_t}{p_t} \right) \Sigma  \right] \\
    &= -\nabla \cdot \left( \lambda_t f \right)  + \frac{1}{2} \tr\left[ \Sigma \left( p \nabla^2 \mu - \mu \nabla^2 p \right) \right].
\end{align*}

We now simplify the second-order term using the identity:
\[
p \nabla^2 \mu - \mu \nabla^2 p = \nabla \cdot (p \nabla \mu - \mu \nabla p)
\]

Hence $\tr\left[ \Sigma (p \nabla^2 \mu - \mu \nabla^2 p) \right]
= \nabla \cdot \left( \Sigma (p \nabla \mu - \mu \nabla p) \right) $. 

So the adjoint equation becomes ($\mu = \lambda /p$) is 
\begin{align*}
     \partial_t \lambda_t(x) &= -p_t(x) D_{p_t}u_\mathrm{int}(x,p_t(x)) -\nabla \cdot \left(\lambda_t(x) f(x) + \frac{1}{2} \Sigma \left(p_t(x) \nabla \frac{\lambda_t(x)}{p_t(x)} - \frac{\lambda_t(x)}{p_t(x)} \nabla p_t(x)\right) \right)\\
     \intertext{then we use $p_t(x) \nabla \frac{\lambda_t(x)}{p_t(x)} = \nabla \lambda_t(x) - \lambda_t(x) \nabla \log p_t(x)$, so we get}
     &= -p_t(x) D_{p_t}u_\mathrm{int}(x,p_t(x)) -\nabla \cdot \left(\lambda_t(x) [f(x) -\Sigma \nabla \log p_t(x)] + \frac{1}{2} \Sigma \nabla \lambda_t(x) \right)
\end{align*}
\end{proof}
\subsection{Algorithm}
    We can summarise computation of the gradient of the utility functional in the setting of lemma \ref{sec:main_lemma} as follows:
    \begin{enumerate}
        \item Set a sensor schedule $\xi$ and a number of Monte Carlo iterations.
        \item For each $n=1,\ldots, N$:
        \begin{enumerate}
        \item Simulate independent signal paths $X^{(n)}(t)$ and corresponding ($\xi$-dependent) observational paths $Z^{(n)}(t)$.
        \item Solve the filtering problem by computing the filtering distribution $p^{(n)}_t$.
        \item Solve the backward adjoint equation for $\lambda_t^{(n)}$.
        \item Evaluate the integral inside the expectation in \eqref{eq:frechet_der}, call this $\eta^{(n)}$
        \end{enumerate}        
        \item Set $\eta(t) = \frac{1}{N}\sum_{n=1}^N \eta^{(n)}(t)$ as a Monte Carlo approximation of the true gradient.
    \end{enumerate}

\begin{rem}\label{rem:SGD}
    Choosing a small value of $N$ (rather than letting $N\to \infty$) leads to a stochastic gradient approximation similar to the ideas of stochastic gradient descent \cite{kiefer1952stochastic,latz2021analysis} and sample average approximation \cite{kim2014guide,kleywegt2002sample}.
\end{rem}

\section{Numerical experiments}\label{sec:numerics}
We now briefly demonstrate the feasibility of the proposed approach: We will see that in our guiding example we indeed can optimise the sensor schedule to arrive at much higher realised information gains from data, and we show the same behaviour for a linear and Gaussian example. All our examples are very low-dimensional, since solving the (log-)Zakai equation becomes unfeasible for high dimensions. We refer to section \ref{sec:conclusion} for a comment on how to make this approach more viable in realistic application settings.
\subsection{Guiding example: Logistic Growth}
\greybox{Since $X(t) = x_0$ is constant, the generator is $\cL = 0$. The observations are given by
\begin{equation*}
    \dd Z(t) = g(X(t),t)\xi(t)\dd t + \gamma\sqrt{\xi(t)}\dd W_t
\end{equation*}
for $t\in [0,T]$. The log-Zakai equation is $\dd \log p_t(x) = \frac{1}{\gamma^2}g(x,t)\dd Z(t) - \frac{1}{2\gamma^2}g(x,t)^2\xi(t)\dd t$. In this particular example, the adjoint equation is constant, and we have for all $t \in [0,T]$
\begin{align*}
    \lambda_t(x) = \lambda_T(x) = q_T(x)\left[\log \frac{q_T(x)}{q_0(x)} -\int \log \frac{q_T(y)}{q_0(y)}q_T(y)\dd y\right],
\end{align*}
with the gradient becoming 
\begin{align*}
    \delta_\xi V(t) &= \E\frac{1}{2\gamma(t)^2}\int \lambda_t(x)\ g(x,t)(g(x,t) - 2g(x_0,t)) \dd x
\end{align*}
We test this example by setting $T = 6$, an initial sensor schedule $\xi_0 = \mathrm{Unif}[0,6]$, and performing a (projected) gradient descent in $\xi$, i.e. we set $\xi_{i+1} = \Pi (\xi_i -  \delta_{\xi_i} V)$, where $\Pi$ is a projection operator making sure that the new sensor schedule is a valid probability density. When approximating the gradient, we choose just one Monte Carlo sample, see remark \ref{rem:SGD}. We just do 15 steps to demonstrate how it works (in all figures where there is a blue-to-red gradient, blue corresponds to the initial state $n=0$, and dark red is $n=15$, with the colors interpolating inbetween). The top left figure shows how the sensor schedule $\xi$ changes, over the course of 15 iterations, from a uniform density $\xi_0$ to a more peaked distribution near $t=3$. This is consistent to the frequentist result we obtained in the very beginning, where we got $\tau = \frac{\log(z_0^{-1}-1)}{x_0}\approx 3$. The top right figure shows how the observation process $Z(t)$ responds to the various sensor schedules $\xi$. Bottom left we can see how the final filtering distribution (the posterior distribution for the unknown parameter $x_0$ does in fact show a much larger contraction with respect to the prior (dashed line) if we choose $\xi_15$ (which focusses on observing close to time $t=3$) rather than $\xi_0$ (which is a uniform sensor schedule over the whole observation window).

\includegraphics[width=0.5\textwidth]{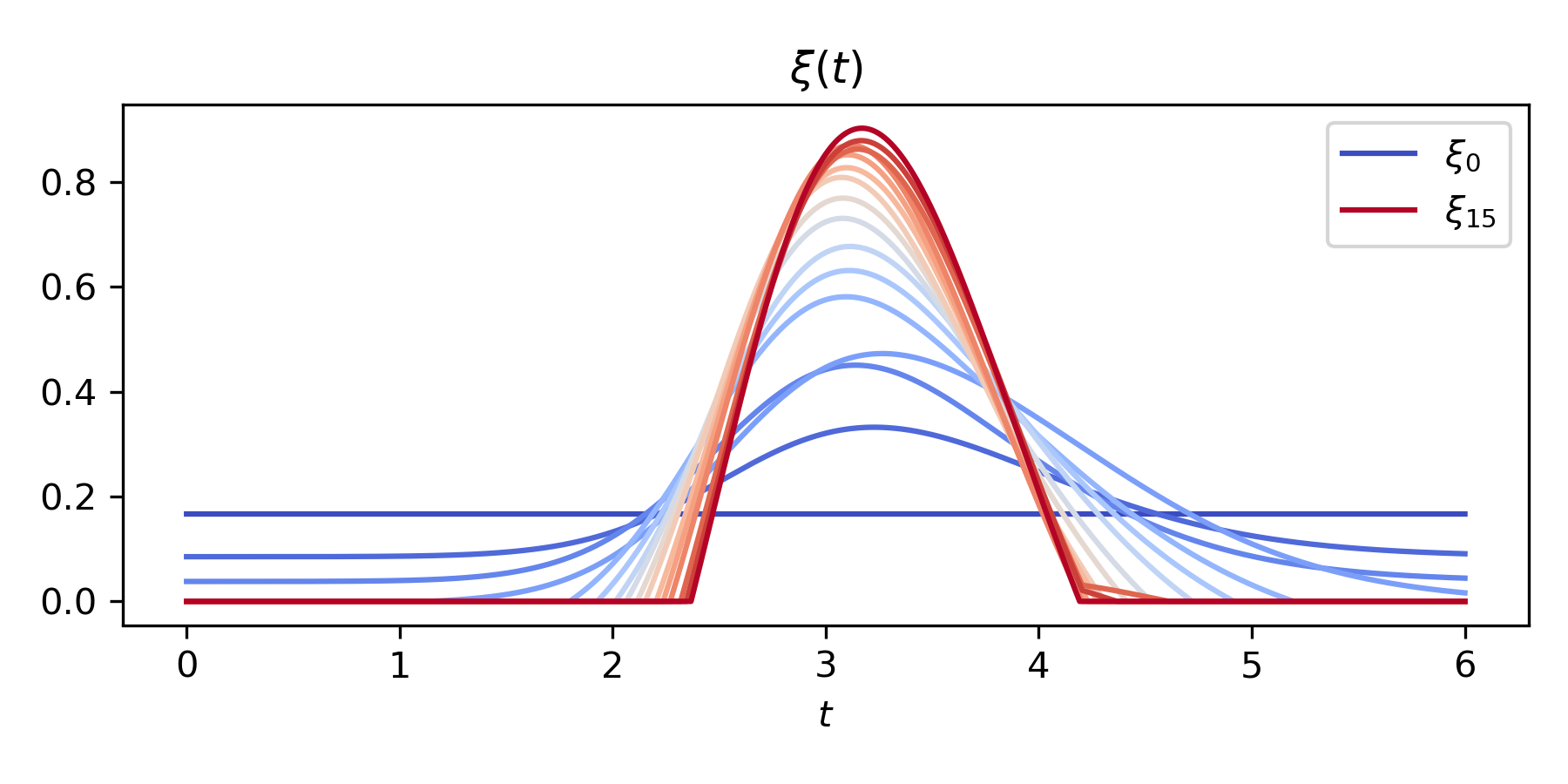}%
\includegraphics[width=0.5\textwidth]{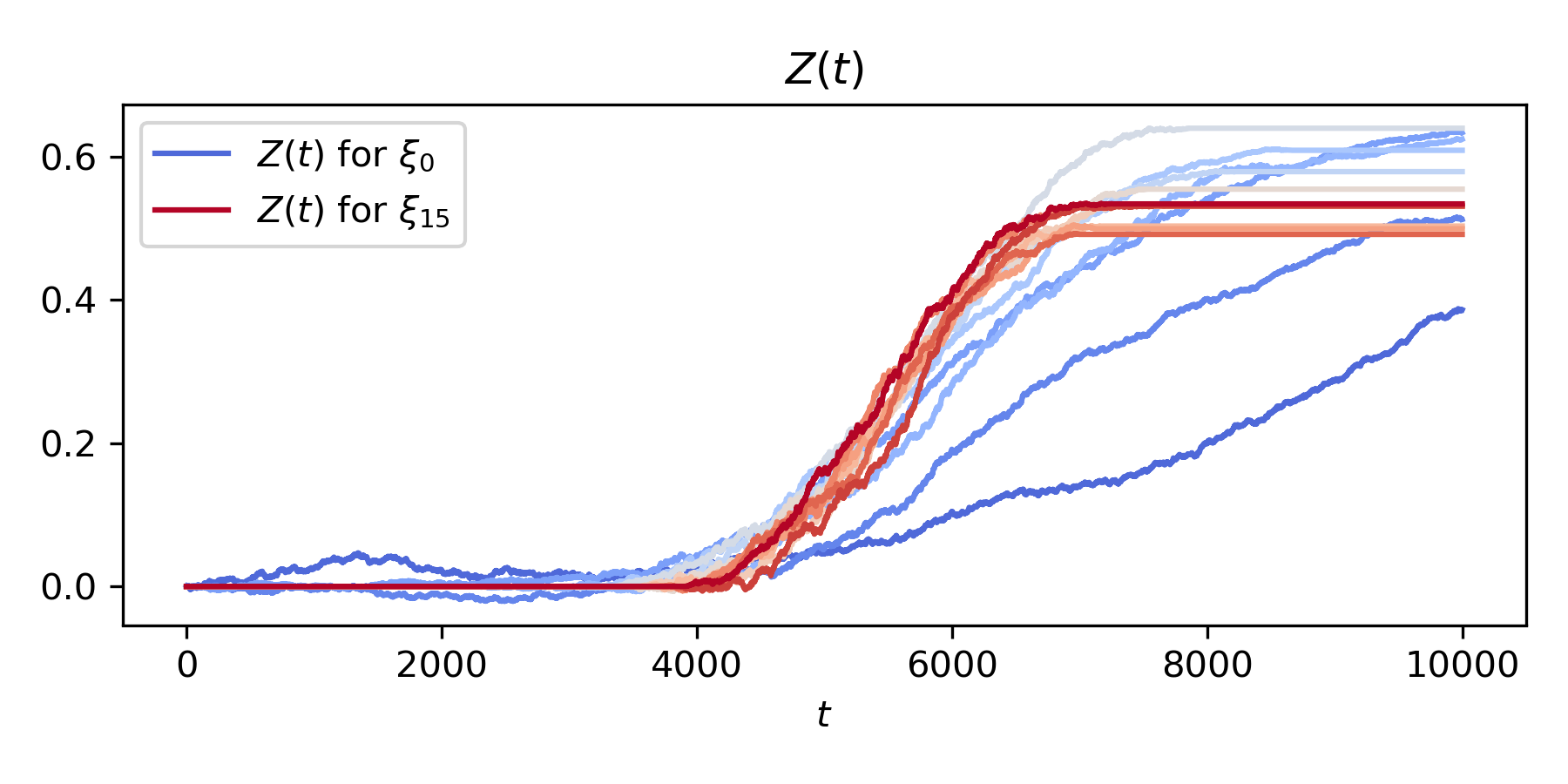}
\includegraphics[width=0.5\textwidth]{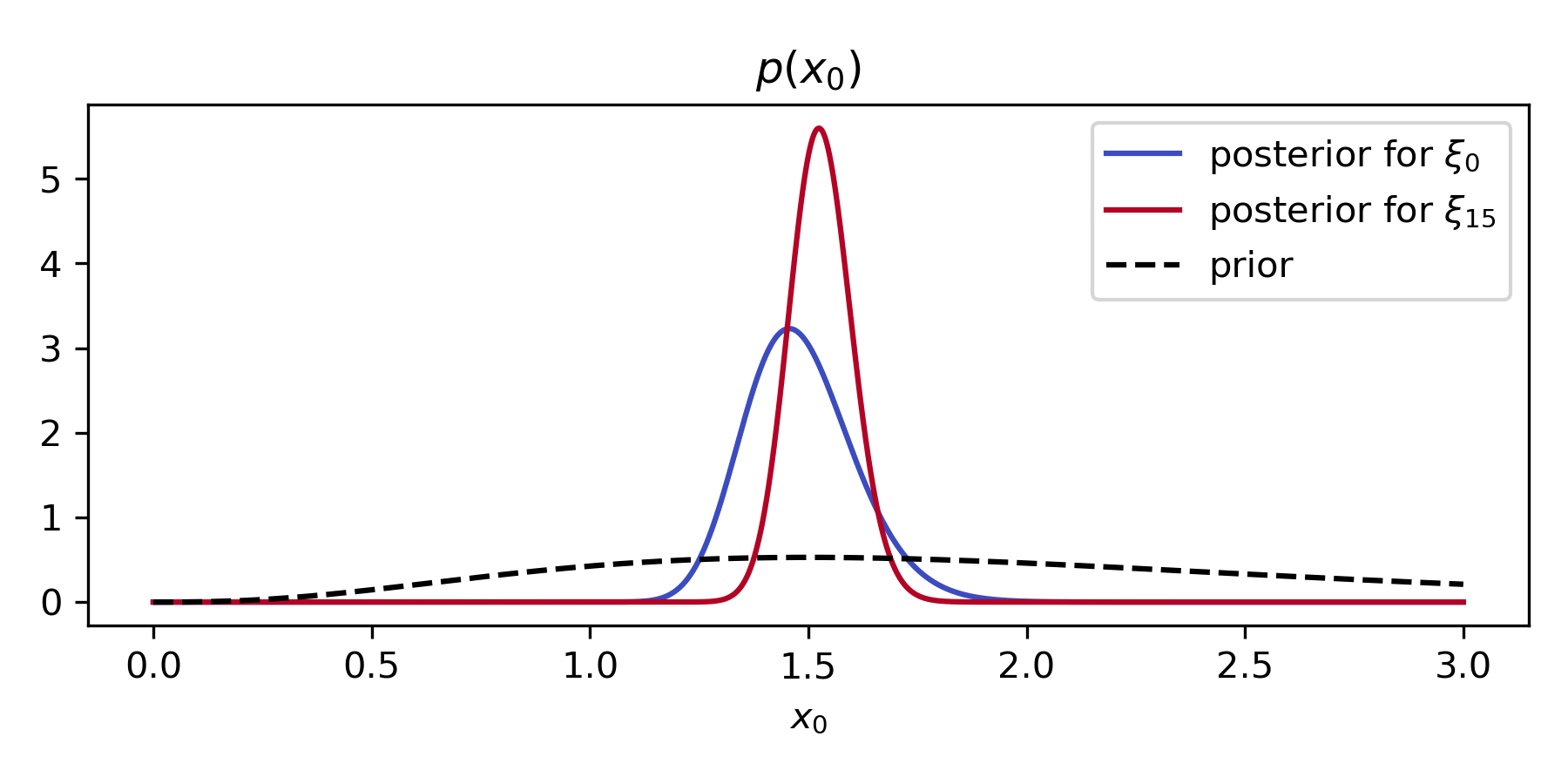}%
\includegraphics[width=0.5\textwidth]{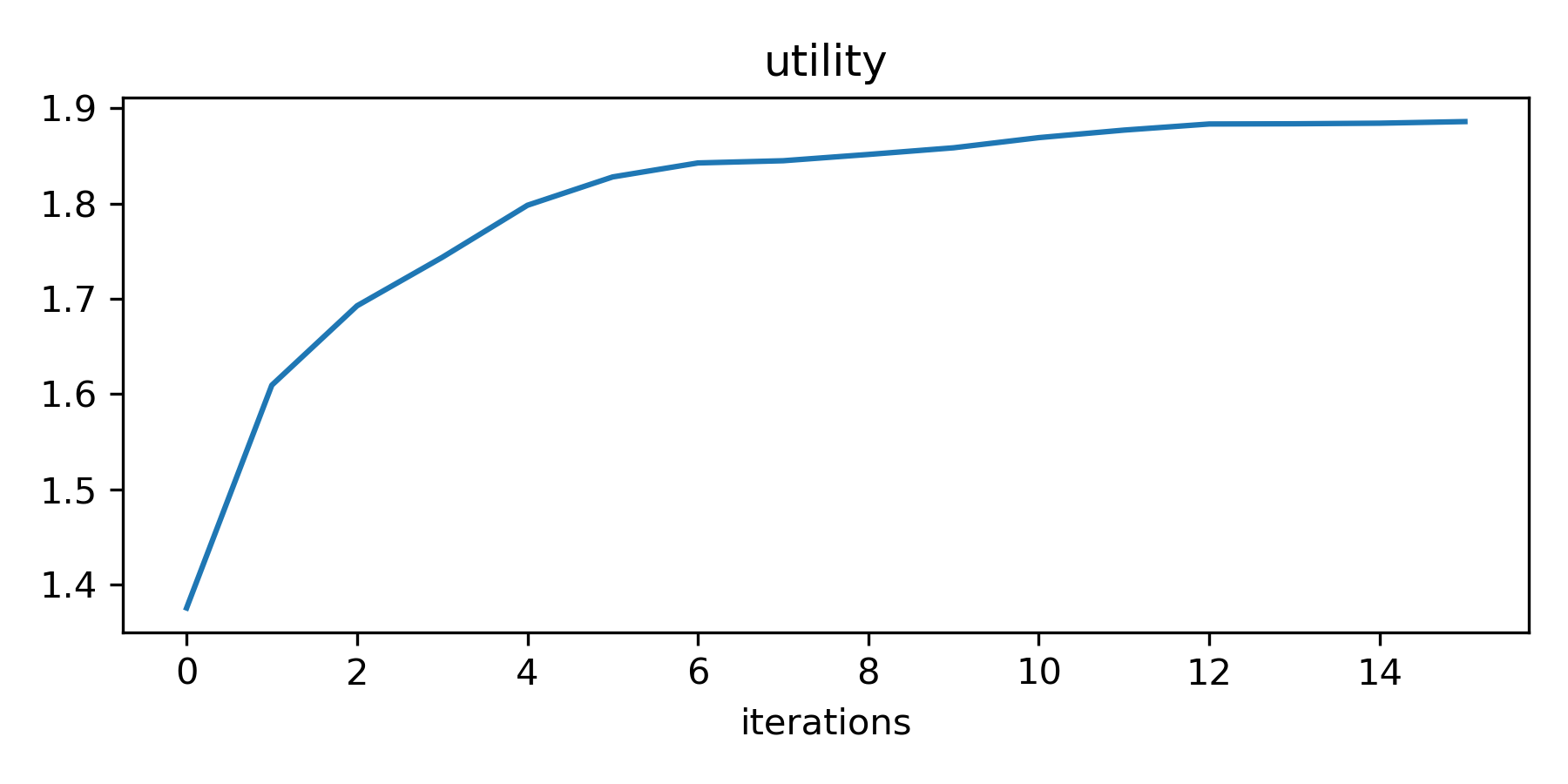}
}
\subsection{Linear Gaussian example}\label{sec:2dexample}
We consider a simple case of observing a two-dimensional Brownian motion $X(t) = (X_1(t),X_2(t))$, where the observation operator selectively picks up only one dimension at a time:
\begin{align*}
    \dd X_1(t) &= \sigma \dd W_1(t)\\
    \dd X_2(t) &= \sigma \dd W_2(t)\\
    \dd Z(t) &= H(t)\cdot X(t) \dd\xi(t) + \gamma \sqrt{\xi(t)}\dd B(t)
\end{align*}
where $t\in [0,6]$ and $H(t) = (1, 0)$ for $t < 3$ and $H(t) = (0,1)$ for $t > 3$. In contrast to the nonlinear setting\footnote{where the utility only penalised uncertainty at the final timestep.}, here we choose an utility depending on the total uncertainty integrated over the whole time horizon $U(p) = \int_0^6 \tr(C(t))\dd t$, where $C$ is the filtering covariance obtained from the Kalman-Bucy filter. Minimising this quantity means trying to have uniformly low variance in both dimensions over the whole observation time frame $t\in[0,6]$.

By optimising this functional using the techniques explained in this manuscript, we see that optimising for $\xi$ leads to a sensor schedule $\xi$ which is strongly concentrated just after $t=0$ and just after $t=3$. Intuitively, an optimal sensor schedule has short bursts of high-SNR observation as early as possible at points in time where information about the first (and second, respectively) dimension becomes available, which keeps down total integrated filtering covariance trace.
\begin{figure}
    \centering
    \begin{subfigure}[t]{0.45\textwidth}
    \includegraphics[width=\linewidth]{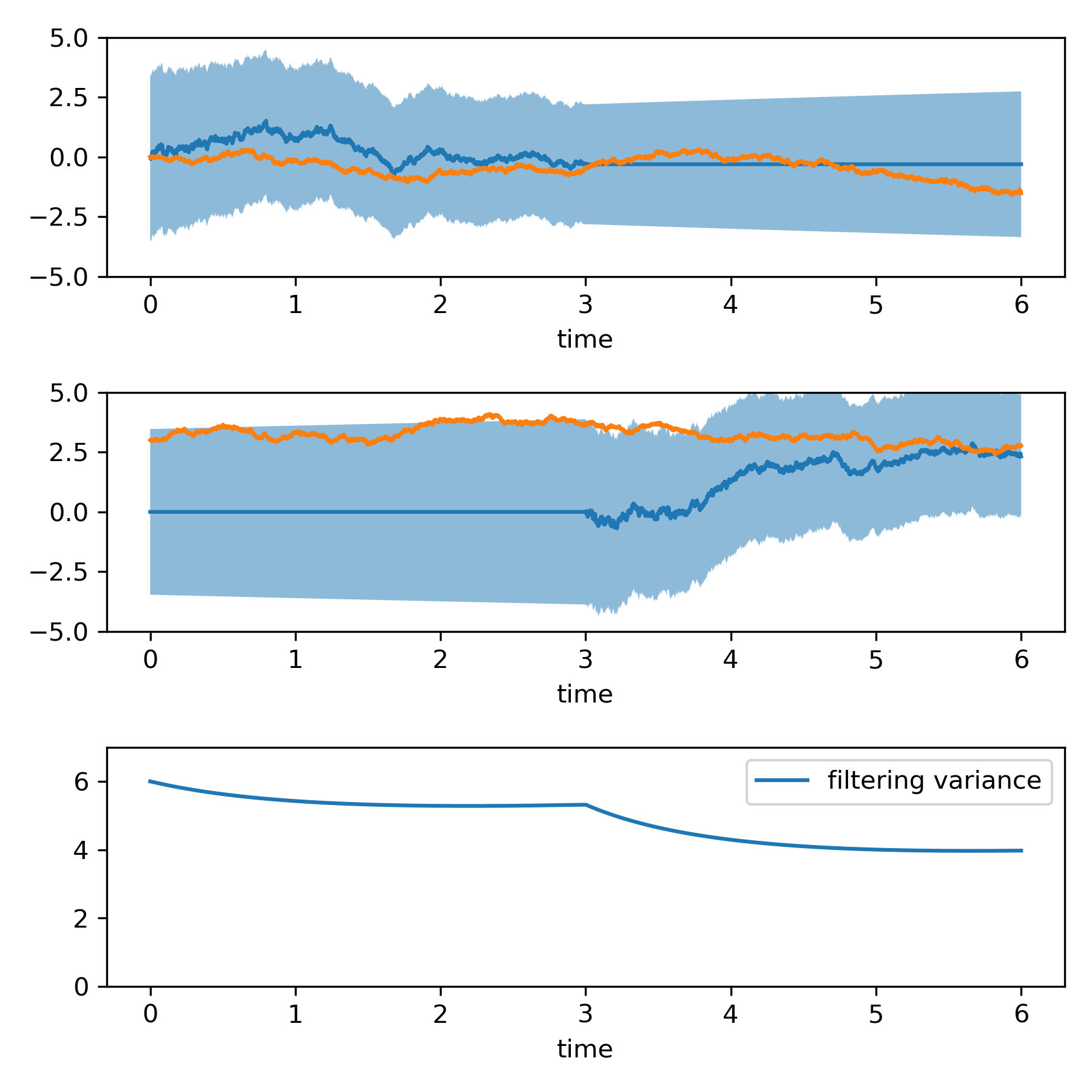}%
    \caption{For $\xi = \mathrm{Unif}[0,6]$: True path (orange) and reconstructed mean $\pm$  two standard deviations from filtering distribution, as well as $\tr(C(t))$ over time.}
    \end{subfigure}%    
    \hfill
    \begin{subfigure}[t]{0.45\textwidth}
    \includegraphics[width=\linewidth]{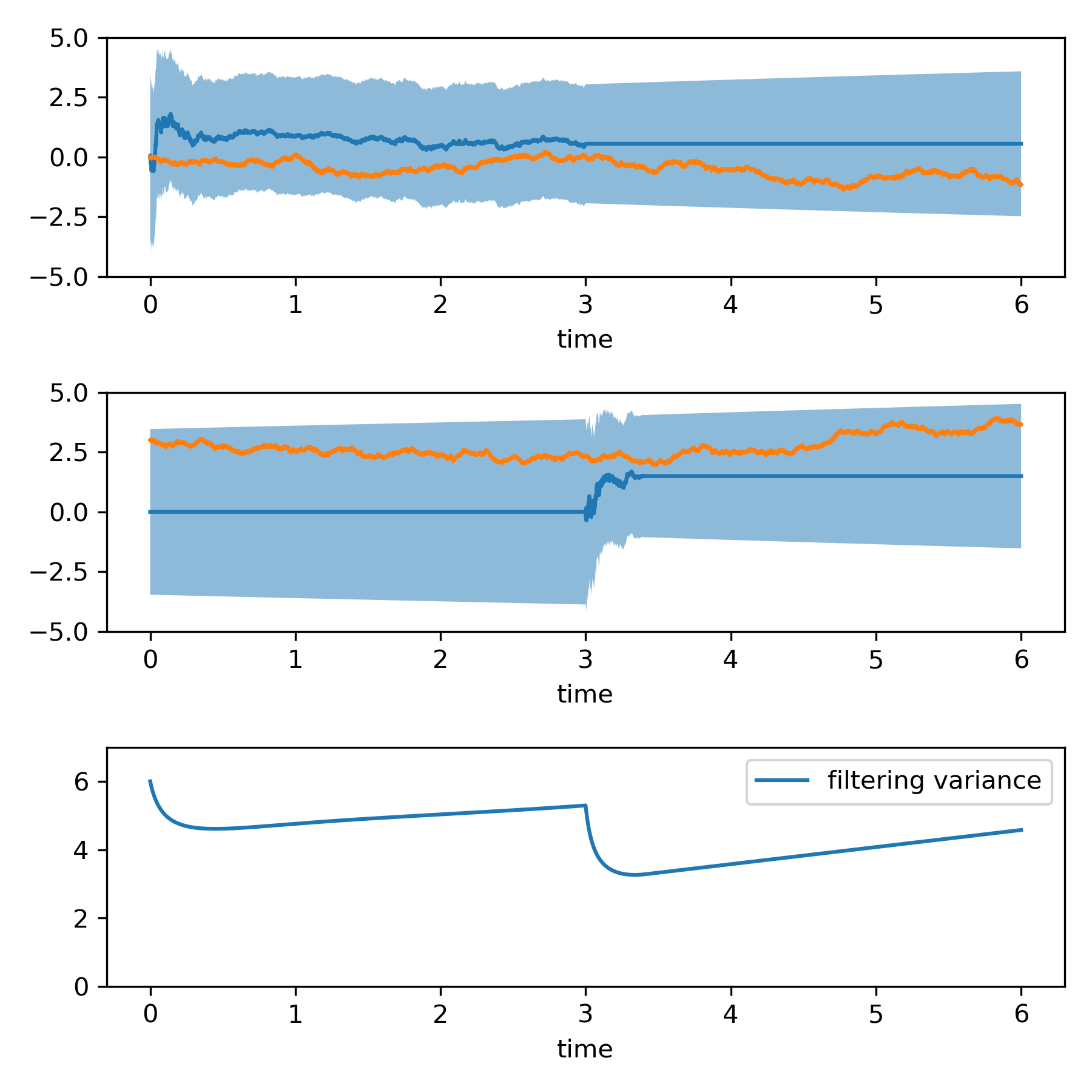}%
    \caption{For $\xi$ obtained from optimisation procedure: True path (orange) and reconstructed mean $\pm$ two standard deviations from filtering distribution, as well as $\tr(C(t))$ over time.}
    \end{subfigure}\\    
    \begin{subfigure}[t]{0.45\textwidth}
    \includegraphics[width=\linewidth]{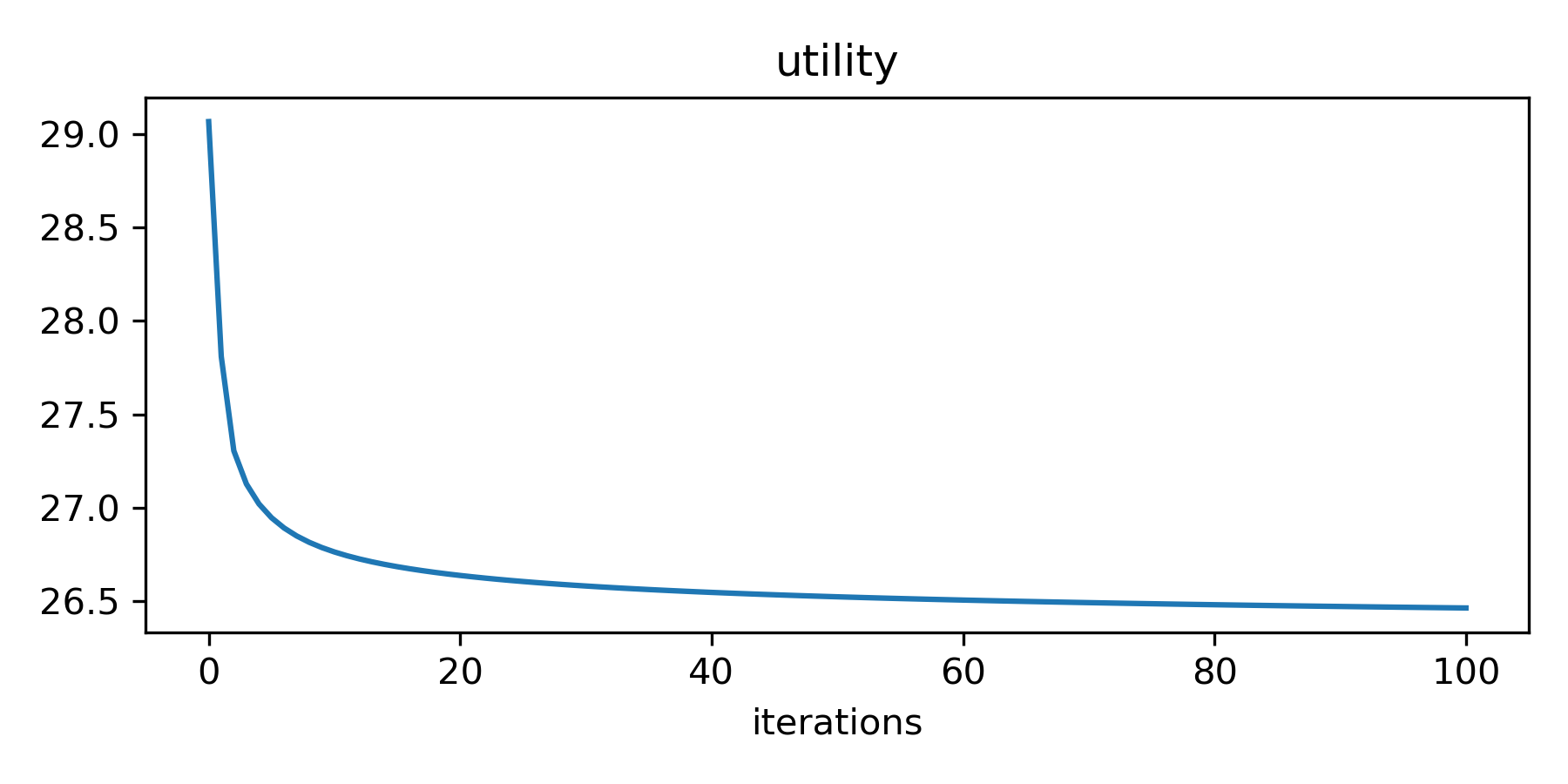}%
    \caption{Value of utility, over the course of 100 minimisation iterations. The fact that this decreases can also be observed visually from comparing the two plots of the (trace of the) filtering variance, seeing that the  area under the first curve is larger than under the second curve. Note that we minimise the utility here, since we are looking for minimal variance.}
    \end{subfigure}%    
    \hfill
    \begin{subfigure}[t]{0.45\textwidth}
    \includegraphics[width=\linewidth]{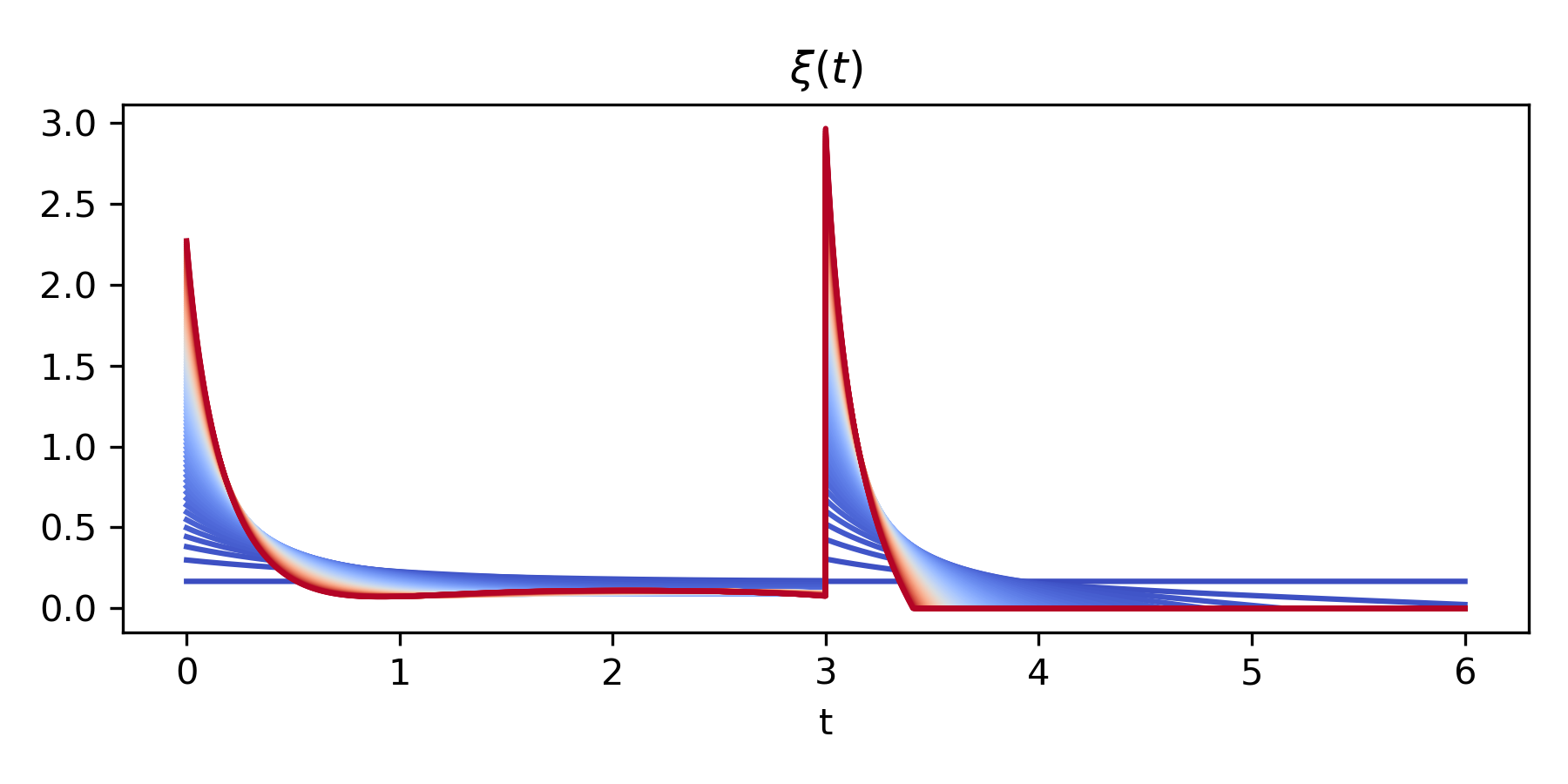}%
    \caption{Evolution of $\xi$ over the course of 100 minimisation iterations: Starts with $\xi_0 = \mathrm{Unif}[0,6]$ (dark blue), and becomes sharply peaked around times $0$ and $3$.}
    \end{subfigure}
    \caption{Visualisation of results of section \ref{sec:2dexample}}
    \label{fig:Simulation_2d_linear}
\end{figure}

\section{Conclusion}\label{sec:conclusion}
We have shown how the problem of optimal sensor placement can be applied to continuous stochastic filtering, and how the gradient of an utility functional can be recovered using the adjoint method. We demonstrated in a few simple cases that this allows us to find provably optimal experimental designs. There are several avenues for future work, e.g.,
\begin{itemize}
    \item Rather than performing a simple projected gradient descent (which ignores the fact that the utility function is an expectation of a random variable), optimisation algorithms more adapted to this stochastic optimisation setup can be explored, such as the SPSA algorithm \cite{spall1987stochastic}.
    \item Embedding probability distributions in a manifold structure allows development of gradient flow algorithms based on Wasserstein or Rao-Fisher flow. The former would allow application of particle methods.
    \item In practice, the Zakai (or Kushner-Stratonovich equation) is rarely solved computationally, instead, tools like the particle filter are used. Extending the theory in the manuscript to developing an adjoint expression for the gradient of a utility functional depending on a particle filter instead of the Zakai equation would be a very interesting avenue for future research.
    \item We started with the problem of finding a single (or several finitely many) observation points, and then relaxed this setting to smooth sensor schedules $\xi$ in the form of a probability distribution. At this point it is unclear how optimality of the smooth density translates into optimality of discrete observation points (e.g., obtained via sampling from this $\xi$).
    \item We focussed on finding the optimal sensor distribution in advance, but of course incoming data will continuously provide more information, so \textit{adaptive} sensor placement on the fly has the potential to improve on this (with the limitation that computing time has to keep up with the speed of incoming data).
    \item A large class of interesting sensor placement problems are built on spatial domains, with sensors being placed at spatial locations (rather than ordered points in time). Generalising the current theory to the spatial setting would be very interesting.
\end{itemize}

\begin{appendix}
    \section{Useful lemmata and statements}
\begin{lem}
    For compatible matrices $A,B,C$,
    \begin{equation}
        \langle A,BC\rangle = \langle AC^\top, B\rangle = \langle B^\top A, C\rangle
    \end{equation}
    \begin{proof}
        Follows from definition $\langle U, V\rangle = \tr(U^\top V)$, cyclic property of the trace, and invariance of the trace under transposition:
        \begin{align*}
            \langle A, BC\rangle &= \tr(A^\top BC) = \tr(CA^\top B) = \tr((AC^\top)^\top B) = \langle AC^\top, B\rangle
        \end{align*}
        and
        \begin{align*}
            \langle A, BC\rangle &= \tr(A^\top BC) = \tr(C^\top B^\top A) = \langle C, B^\top A\rangle = \langle B^\top A, C\rangle
        \end{align*}
    \end{proof}
\end{lem}
\begin{lem}[Summation by parts]
    \begin{align*}
        \sum_{i=1}^N a_i (b_i - b_{i-1}) &= \sum_{i=1}^N a_i b_i - \sum_{i=0}^{N-1} a_{i+1}b_i = a_Nb_N + \sum_{i=1}^{N-1}(a_i - a_{i+1})b_i- a_1b_0\\
        &= a_Nb_N-a_1b_0 - \sum_{i=1}^{N-1}(a_{i+1}-a_i)b_i\\
        &= a_Nb_N-a_1b_0 - \sum_{i=2}^{N}(a_{i}-a_{i-1})b_{i-1}
    \end{align*}
    and
    \begin{align*}
        \sum_{i=0}^{N-1} a_i(b_{i+1}-b_i) = a_{N-1}b_N - a_0b_0 - \sum_{i=1}^{N-1}(a_i-a_{i-1})b_i
    \end{align*}
\end{lem}
\begin{lem}[Derivative of inverse matrix]
    $\dd M^{-1} = -M^{-1}\cdot \dd M \cdot M^{-1}$
\end{lem}
\begin{lem}[Matrix derivatives]\label{lem:matrix_derivatives}
    \begin{align*}
        \partial_M \tr(M) &= I\\
        \partial_M \log\det(M) &= 2M^{-1} - M^{-1}\circ I.
    \end{align*}
\end{lem}

\section{Recap of the adjoint method for computing the gradient of a utility function depending on a trajectory}
\label{sec:recap_adjoint}
We consider a function $F(x) = f(x(T))$, where $x$ is the solution of some ODE depending on some parameter $p$, i.e. $h(x,\dot x, p) = 0$ for all $t\in[0,T]$, and $x(0)=x_0$ (we assume no dependence of $x(0)$ on $p$).

The corresponding Lagrangian is 
\begin{align}
    \mathbf L(x,\dot x, p) = f(x(T)) + \int_0^T  \lambda(t)^\top h(x(t),\dot x(t), p)\dd t + \mu^\top (x(0)-x_0)
\end{align}
and we can compute its variation with respect to $p$ as
\begin{align}
    \DD_p\mathbf L(x,\dot x, p) &= \partial_x f(x(T)) \DD_p x(T) + \int_0^T \lambda^\top \left(\partial_x h \DD_px + \partial_{\dot x} \DD_p\dot x + \partial_p h \right)\dd t\\
    &= \partial_x f(x(T)) \DD_px(T) + \int_0^T \lambda^\top \partial_x h \DD_px \dd t+\int_0^T\lambda^\top \partial_{\dot x} h \DD_p\dot x \dd t +\int_0^T \lambda^\top\partial_p h \dd t\\
    &= \partial_x f(x(T)) \DD_p x(T) + \int_0^T \lambda^\top \partial_x h \DD_px \dd t ~+\\
    &+[\lambda^\top(t) \partial_{\dot x}h\DD_p  x(t)]_0^T - \int_0^T(\dot \lambda^\top \partial_{\dot x}h + \lambda^\top d_t \partial_{\dot x} h) \DD_px\dd t+\int_0^T \lambda^\top\partial_p h \dd t\\
    &=(\partial_x f(x(T)) + \lambda(T)^\top \partial_{\dot x}h)\DD_px(T) - \lambda(0)^\top \partial_{\dot x}h \DD_px(0)\\
    &+ \int_0^T \left[\lambda^\top \left(\partial_x h - d_t\partial_{\dot x}h \right) - \dot \lambda^\top \partial_{\dot x}h \right] \DD_px + \lambda^\top \partial_p h \dd t
\end{align}
We now require $\lambda$ to be such that the integrand vanishes. This leads to the adjoint ODE
\begin{align*}
    \dot \lambda(t)^\top \partial_{\dot x}h(x(t),\dot x(t),p(t)) &= \lambda(t)^\top(\partial_x h(x(t),\dot x(t),p(t)) - d_t \partial_{\dot x}h(x(t),\dot x(t),p(t)))\\
    \lambda(T)^\top \partial_{\dot x}h(x(T),\dot x(T),p(T)) &= -\partial_x f(x(T))
\end{align*}
so that the gradient is given by
\begin{align*}
    \DD_pF(x) = \int_0^T \lambda(t)^\top \partial_p h(x(t),\dot x(t), p(t))\dd t
\end{align*}
This means that we can evaluate the Fr\'echet derivative of $F$ with respect to $p$ by solving a backwards-in-time ``adjoint ODE'' for a Lagrangian multiplier function $\lambda$, and then evaluation certain integrals depending on $\lambda$.
\section{Time rescaling does not yield a scheduled filtering problem}\label{sec:rescale}
In this section we quickly study the question whether the scheduled filtering problem \eqref{eq:filtering_cts} can be viewed as a simple time rescaling of a standard observation process (and we answer this to the negative). Given a so-called base process, 
\begin{align*}
    \dd Y(s) = g(Y(s))\dd s + \sigma \dd B(s)  
\end{align*}
construct a time changed process $Z$ such that
\begin{align*}
    Z(t) = Y(r(t))
\end{align*}
where $r(t)$ may be either a deterministic or stochastic time change.  In our context, taking $r(t)$ to be a deterministic time change makes the most sense, i.e. $r(t)$ such that $\frac{dr(t)}{dt} = \xi(t)$ with $\xi$ a deterministic function of time.  Then 
\begin{align*}
    \dd Z(t) = \dd Y(r(t)) &= g(Y(r(t)))\dd r(t) + \sigma \dd B(r(t)) \\
    & = g(Z(t))\xi(t)\dd t + \sigma \sqrt{\xi(t)}\dd W(t) 
\end{align*}
where $W(t)$ is a standard Brownian motion in time scale $t$ independent from $B(t)$. In our case, we instead want to consider a time change of the observation process, which itself depends on another stochastic process $X$, i.e.
\begin{align*}
    \dd Y(s) = g( X(s))\dd s + \sigma \dd B(s)
\end{align*}
doing a deterministic time change in a similar way yields
\begin{align*}
    \dd Z(t) &= \dd Y(r(t)) = g(X(r(t)))\dd r(t) + \sigma \dd B(r(t)) \\
    & = g(X(r(t))) \xi(t) \dd t +  \sigma \sqrt{\xi(t)} \dd W(t) 
\end{align*}
which is not consistent with \eqref{eq:filtering_cts} as the signal process is evaluated at time $r(t)$ rather than $t$. 
\end{appendix}

\bibliographystyle{alpha}
\bibliography{lit.bib}

\end{document}